\documentclass[12pt]{amsart}
\usepackage[a4paper, total={156mm, 224mm}, centering]{geometry}
\numberwithin{equation}{section}

\usepackage{amssymb}
\usepackage{mathtools}
\usepackage{microtype}
\usepackage{hyperref,hyphenat}
\hypersetup{hidelinks}

\allowdisplaybreaks

\newtheorem{theorem}{Theorem}[section]
\newtheorem{proposition}[theorem]{Proposition}
\newtheorem{lemma}[theorem]{Lemma}
\newtheorem{corollary}[theorem]{Corollary}
\newtheorem{conjecture}[theorem]{Conjecture}
\theoremstyle{definition}

\newtheorem{question}[theorem]{Question}
\theoremstyle{remark}
\newtheorem{remark}[theorem]{Remark}

\newcommand{\A}{\mathbb{A}}
\newcommand{\C}{\mathbb{C}}
\newcommand{\Q}{\mathbb{Q}}
\newcommand{\R}{\mathbb{R}}
\newcommand{\G}{\mathbb{G}}
\newcommand{\Z}{\mathbb{Z}}
\newcommand{\sM}{\mathcal{M}}
\newcommand{\sO}{\mathcal{O}}
\newcommand{\sC}{\mathcal{C}}
\newcommand{\sI}{\mathcal{I}}
\newcommand{\sX}{\mathcal{X}}

\newcommand{\fm}{\mathfrak{m}}
\newcommand{\Pone}{\mathbb{P}^{1}}
\newcommand{\Rat}{\mathrm{Rat}}
\newcommand{\Aff}{\mathrm{Aff}}
\newcommand{\Conf}{\mathrm{Conf}}
\newcommand{\Fix}{\mathrm{Fix}}
\newcommand{\PGL}{\mathrm{PGL}}

\title[Generic reconstruction from periods one and two]{Generic reconstruction of rational maps from multipliers of periods one and two}

\author{Geng-Rui Zhang}
\address{School of Mathematical Sciences, Peking University, Beijing 100871, China}
\email{grzhang@stu.pku.edu.cn}
\email{chibasei@163.com}

\date{September 6, 2026}


\begin{document}

\begin{abstract}
For every integer $d\geq2$ and every field of characteristic different from $2$, we prove that on the moduli space $\sM_d$ of degree-$d$ rational maps, the multiplier spectrum morphism formed from the periodic points of periods one and two is birational to the closure of its image. Consequently, over every algebraically closed field of characteristic different from $2$, it is generically injective, which proves a recent conjecture of Ji and Xie in characteristic zero. The proof uses a fixed-index normal form, a non-archimedean degeneration of two-cycles, and birational reconstruction of an affine fixed-point configuration from pair invariants.
\end{abstract}

\maketitle
\tableofcontents

\section{Introduction}
Throughout the paper, $k$ denotes an algebraically closed field of characteristic different from $2$.

\subsection{Multiplier spectra and the main theorem}
Fix an integer $d\geq2$. Let $f:\Pone_k\to\Pone_k$ be a rational map of degree $d$. If $z_0\in\Pone(k)$ is $f$-periodic of exact period $m\geq1$, the \emph{multiplier of $f$ at $z_0$} is
\[
\rho_f(z_0):=d(f^{\circ m})(z_0)\in k,
\]
where the differential is viewed as a scalar on the one-dimensional tangent space $T_{z_0}\Pone_k$. It is invariant under conjugacy by $\PGL_2(k)$.

Let $\Rat_d$ be the open subscheme of $\mathbb{P}^{2d+1}_{\Z}$ parametrizing pairs of degree-$d$ homogeneous forms with nonzero resultant. The group scheme $\mathrm{SL}_2$ acts on $\Rat_d$ by conjugacy. Silverman proved that the geometric quotient
\[
\sM_d:=\Rat_d/\mathrm{SL}_2
\]
exists as an affine integral scheme over $\Z$; see \cite{Sil98}. We denote its base change to a field $K$ by $\sM_{d,K}$. If $K$ is algebraically closed, then $\sM_{d,K}$ is the coarse moduli space of degree-$d$ rational maps over $K$ and is an irreducible affine variety of dimension $2d-2$; see \cite[Theorem~4.36(c)]{Sil07} and \cite[Chapter~2]{Sil12}. We write $[f]\in\sM_{d,k}(k)$ for the conjugacy class of $f$.

For a rational map $h:\Pone_k\to\Pone_k$, let $\Fix(h)$ denote its fixed-point scheme, namely the scheme-theoretic intersection of the graph of $h$ with the diagonal in $\Pone_k\times\Pone_k$. If $h$ has degree $e\geq2$, then the zero-dimensional scheme $\Fix(h)$ has length $e+1$. We also view $\Fix(h)$ as a multiset of cardinality $e+1$, with elements in $\Pone(k)$. The phrase ``simple fixed point'' means multiplicity one in $\Fix(f)$, equivalently multiplier different from $1$.

For $n\geq1$, let $\mathrm{Per}_n(f)$ be the multiset of fixed points of $f^{\circ n}$, with multiplicities in $\Fix(f^{\circ n})$, and set
\[
N_{d,n}:=d^n+1.
\]
Let
\[
S_n(f)=\left(\sigma_{1,n}(f),\ldots,\sigma_{N_{d,n},n}(f)\right),
\]
where $\sigma_{j,n}(f)$ is the $j$th elementary symmetric function of the $N_{d,n}$ multipliers
\[
\left\lbrace\rho_{f^{\circ n}}(z)\colon z\in\mathrm{Per}_n(f)\right\rbrace.
\]
These symmetric functions are regular conjugacy invariants and define a morphism over $\Z$; see \cite[\S~4]{Sil98} and \cite[\S~4.5]{Sil07}. For a field $K$, we denote its base change by
\[
S_{n,K}:\sM_{d,K}\to\A_K^{N_{d,n}};
\]
see \cite[\S~1.1]{JX25} and \cite[\S~1.1]{Zha26}. Equivalently, the data of $S_n(f)$ is encoded by
\begin{equation}\label{eq:mulpoly}
\Phi_{n,f}(T):=\prod_{z\in\Fix(f^{\circ n})}\left(T-\rho_{f^{\circ n}}(z)\right).
\end{equation}
For $n\geq1$ and a field $K$, the \emph{multiplier spectrum morphism up to level $n$} is
\[
\tau_{d,n,K}:=(S_{1,K},\ldots,S_{n,K}):\sM_{d,K}\to \A_K^{N_{d,1}}\times\cdots\times\A_K^{N_{d,n}}.
\]
When the base field is clear, we often omit the subscript $K$.

For $x\in\sM_{d,k}(k)$, the morphism $\tau_{d,n,k}$ is called \emph{injective at $x$} if
\[
\tau_{d,n,k}^{-1}\left(\tau_{d,n,k}(x)\right)=\{x\}.
\]
It is injective on a subset if it is injective at every point of that subset, and it is \emph{generically injective} if it is injective on a nonempty Zariski open subset of $\sM_{d,k}$. This definition concerns the full fiber in the moduli space, not only the restriction of the morphism to the chosen open subset.

Over $\C$, McMullen proved quasi-finiteness of the full multiplier spectrum away from the flexible Latt\`es locus \cite{McM87}. For the case $k=\C$, Ji and Xie proved generic injectivity of the full multiplier spectrum \cite[Theorem~1.3]{JX25}, and also conjectured that $\tau_{d,2,\C}$ is already generically injective \cite[Conjecture~1.8]{JX25}. The cases $d=2$ and $d=3$ of the conjecture of Ji--Xie follow from results of Milnor \cite{Mil93} and Gotou \cite[Theorem~1.2]{Got23}, respectively. Huguin proved the polynomial analogue in a stronger form: for $k=\C$, the multiplier morphism up to level two on the polynomial moduli space is finite and birational onto the closure of its image \cite[Main Theorem]{Hug24}.

Our main result is the positive resolution of the conjecture of Ji and Xie, in all characteristics different from $2$:
\begin{theorem}\label{thm:main}
Let $d\geq2$ be an integer, and let $F$ be a field of characteristic different from $2$. Then the induced dominant morphism
\[
\tau_{d,2,F}:\sM_{d,F}\to\overline{\tau_{d,2,F}(\sM_{d,F})}\subseteq\A_F^{d+1}\times\A_F^{d^2+1}
\]
is birational. Consequently, if $F$ is algebraically closed, then $\tau_{d,2,F}$ is generically injective.
\end{theorem}

\begin{remark}\label{rem:char2intro}
The restriction on the characteristic is specific to the pair-invariant reconstruction used in the proof: the coefficient containing $2B_{ij}$ in the two-cycle expansion vanishes in characteristic $2$ (see Proposition~\ref{prop:exp}). In fact, the framework of the main theorem's proof can be modified in characteristic $2$ by carrying the expansion one term further and working with the squared fixed-point coordinates. We give a sketch of the characteristic-$2$ variant in \S~\ref{subsec:char2}.
\end{remark}

\begin{remark}
The methods of this paper and those of \cite{JX25} are rather different. The proof in \cite{JX25} is more conceptual. Assuming that generic injectivity fails, Ji and Xie construct pairs of algebraic families with the same multiplier spectrum and use a dynamical Andr\'e--Oort type result for two families of rational maps, together with the theory of intertwined maps, to obtain a contradiction. By contrast, our approach is more elementary and computational. We use the fixed-point index relation to obtain an explicit normal form, study degenerations of two-cycle multipliers, and reconstruct the generic fixed-point configuration from the resulting pair invariants. We also note that the argument of \cite{JX25} gives some information on the non-injective locus of the full multiplier spectrum morphism: the one-dimensional families of rational maps with the same multiplier spectrum they constructed are shown to be generically intertwined.
\end{remark}

\begin{remark}
Our methods for rational maps can also be adapted to show the generic injectivity part of Huguin's polynomial theorem \cite{Hug24} (over every algebraically closed field of characteristic different from $2$), although they do not give the properness or finiteness result. Indeed, for a polynomial the fixed point at $\infty$ has multiplier $0$ and fixed-point index $1$, and, writing $Q_x(z)=\prod_{i=1}^d(z-x_i)$, the polynomial locus in the fixed-index parameter space is characterized by the condition that $c_iQ_x'(x_i)$ is independent of $i$. Under the specialization $c_i=tw_i$, this gives $w_i=1/Q_x'(x_i)$ up to a common scalar, and the family used in the proof becomes $f_{x,t}(z)=z-Q_x(z)/t$. For $d=2,3$, generic reconstruction on the polynomial locus already follows from the fixed-point multipliers. For $d\geq4$, the degeneration argument is applied on the polynomial locus, with its nonvanishing conditions checked subject to the relations $w_iQ_x'(x_i)=w_jQ_x'(x_j)$. Thus, the same degeneration and reconstruction argument shows that the multiplier morphism up to level two on the polynomial moduli space is generically injective in characteristic different from $2$.
\end{remark}

\subsection{Idea of the proof}\label{sec:mainidea}
Let $[f]\in\sM_{d,k}(k)$ lie in the Zariski open dense locus on which $\Fix(f^{\circ2})$ is reduced. Thus, every fixed point $z$ of $f^{\circ2}$ satisfies $\rho_{f^{\circ2}}(z)\neq1$, and every fixed point $z$ of $f$ satisfies $\rho_f(z)\neq\pm1$. We will further require the $d+1$ multipliers of the fixed points of $f$ to be pairwise distinct when we pass from unordered to ordered fixed-point data in the final step.

For a rational map $h$ of degree $\geq2$, the \emph{fixed-point index} of a fixed point $z_0$ is
\[
\iota_h(z_0):=\operatorname{Res}_{u=0}\frac{du}{u-\widetilde{h}(u)},
\]
where $u$ is any local coordinate with $u(z_0)=0$ and $\widetilde{h}$ is the local expression of $h$. This residue is independent of the chosen local coordinate. If the fixed point $z_0$ is simple, then
\[
\iota_h(z_0)=\frac{1}{1-\rho_h(z_0)}\in k^\times.
\]
The rational fixed-point index formula is
\[
\sum_{z\in |\Fix(h)|}\iota_h(z)=1,
\]
where $|\Fix(h)|$ denotes the set of distinct fixed points; see \cite[Exercise~1.17]{Sil07}.

Choose and mark one fixed point of $f$, conjugate it to $z_0=\infty$, and order the other fixed points as $z_1,\ldots,z_d\in k$. Write $c_i=\iota_f(z_i)$ for $0\leq i\leq d$.

An \emph{ordered affine $d$-configuration} is an element
\[
x=(x_1,\ldots,x_d)\in\Conf_d(\A^1):=\{(x_1,\ldots,x_d)\in\A^d\colon x_i\neq x_j\text{ for }i\neq j\}.
\]
Two $d$-configurations $x=(x_1,\ldots,x_d)$ and $y=(y_1,\ldots,y_d)$ are \emph{affinely equivalent} if there exists $\sigma\in\Aff$ such that $y_i=\sigma(x_i)$ for $1\leq i\leq d$, where
\[
\Aff=\mathrm{Aut}(\A^1)=\{z\mapsto az+b\colon a\neq0\}
\]
is the automorphism group of $\A^1$. The \emph{space of affine $d$-configurations} is the geometric quotient
\[
\sC_d:=\Conf_d(\A^1)/\Aff.
\]
For $x\in\Conf_d(\A^1)$, we use $[x]$ to denote its class in $\sC_d$. Clearly, the normalization $(x_1,x_2)=(0,1)$ identifies $\sC_d$ with a Zariski open dense subset of $\A^{d-2}$.

An \emph{ordered fixed-point index vector} is a tuple $c=(c_0,c_1,\ldots,c_d)\in(k^\times)^{d+1}$ satisfying $\sum_{i=0}^d c_i=1$, where $c_i$ is attached to the labeled fixed point $z_i$. Proposition~\ref{prop:nf} shows that, for such a $c$, an ordered affine $d$-configuration $x$, and the marked fixed point $z_0=\infty$, the corresponding rational map is
\[
f(z)=z-\frac{1}{R(z)},\quad R(z)=\sum_{i=1}^d\frac{c_i}{z-x_i}.
\]
A M\"obius transformation preserving the marked point at infinity is an affine transformation in $\Aff$. Therefore, for a fixed ordered index vector, conjugacy preserving all labels is exactly affine equivalence of $x$.

The key specialization is algebraic. Let $w_1,\ldots,w_d$ be algebraically independent over $k$, let $t$ be the
uniformizer of the complete Laurent-series field
\[
K_0:=k(w_1,\ldots,w_d)((t)),
\]
and specialize
\[
c_i=t w_i\quad(1\leq i\leq d),\quad c_0=1-t\sum_{i=1}^d w_i.
\]
Here every $w_i$ is a $t$-adic unit. Write
\[
f_{x,t}(z)=z-\left(t\sum_{i=1}^d\frac{w_i}{z-x_i}\right)^{-1}.
\]
For a $d$-configuration $x=(x_1,\ldots,x_d)$, define
\[
T_w(x)_i=\sum_{r\neq i}\frac{w_r}{x_i-x_r},\quad B_{ij}(x):=(x_i-x_j)\left(T_w(x)_i-T_w(x)_j\right).
\]
If all $x_i-x_j$ ($i\neq j$) are $t$-adic units, then for every unordered pair $(i,j)$ ($i<j$), there is a unique two-cycle with one point in the residue disc of $x_i$ and the other in the residue disc of $x_j$. If $p_{ij}$ is either point of this cycle, Proposition~\ref{prop:exp} gives
\[
\rho_{f_{x,t}}(p_{ij})=\frac{1}{t^2w_iw_j}-\frac{1}{t}\left(\frac{1}{w_i}+\frac{1}{w_j}+\frac{2B_{ij}(x)}{w_iw_j}\right)+O(1).
\]
Thus, the reduction of $t^2\rho_{f_{x,t}}(p_{ij})$ is $1/(w_iw_j)$. These elements are pairwise distinct because the $w_i$ are algebraically independent. By \emph{pair-labeling} we mean the unique labeling of the roots of the rescaled two-cycle multiplier polynomial by pairs $(i,j)$ ($i<j$) for which the root labeled $(i,j)$ ($i<j$) reduces to $1/(w_iw_j)$. Applying Lemma~\ref{lem:hensel} in one variable to the rescaled multiplier polynomial lifts each of these distinct residue roots uniquely. This makes the pair-labeling unique throughout the corresponding residue neighborhood. The coefficient of $t^{-1}$ then recovers every $B_{ij}(x)$.

For $u,v\in\A^d(k)$, define a symmetric matrix $M(u,v)\in\mathrm{Mat}_{d\times d}(k)$ by
\[
M(u,v)_{ij}=(u_i-u_j)(v_i-v_j),\quad 1\leq i,j\leq d.
\]
Proposition~\ref{prop:fac} classifies its generic factorizations. Applying that proposition to $u=x$ and $v=T_w(x)$ implies that for generic $w$, the morphism $[x]\mapsto(B_{ij}(x))_{i<j}$ is birational to the closure of its image.

A second issue is boundary degeneration. Suppose another $d$-configuration has the same two-cycle spectrum but some of its points collide after reduction. After making all coordinates integral, choose a cluster $C$ at the largest valuation of a pairwise difference. After rescaling by one difference within $C$, all internal differences then become units, and every coordinate outside $C$ has negative valuation. The pair-labeling still applies to internal pairs. Contributions from labels outside $C$ vanish after reduction, so the subleading multiplier terms force the internal pair invariants to equal the corresponding invariants of the reference configuration. Lemma~\ref{lem:cl} then gives an algebraic identity that a generic reference configuration does not satisfy. This non-archimedean cluster argument excludes all boundary branches. The multivariable Hensel's lemma (Lemma~\ref{lem:hensel}) and the nonvanishing Jacobian of the pair-invariant map then show that the ordered generic fiber has one reduced point. Finally, distinct fixed-point multipliers identify the ordering, and the factorization of the fixed points of $f^{\circ2}$ into fixed points of $f$ and two-cycles descends the ordered conclusion to $\sM_d$. Birationality over an arbitrary field $F$ of characteristic different from $2$ follows by descent from an algebraic closure of $F$.

\subsection{Organization}
Section~\ref{sec:deg} develops the fixed-index normal form, the multivariable Hensel's lemma, and the two-cycle expansion. Section~\ref{sec:pair} proves reconstruction from the pair invariants and the boundary no-escape statement. Section~\ref{sec:rec} proves reconstruction on the ordered cover and Theorem~\ref{thm:main}, and ends with a sketch of the characteristic-$2$ variant. Section~\ref{sec:app} gives applications, conjectures, and open questions.

\section{Fixed indices and the two-cycle degeneration}\label{sec:deg}
\subsection{The fixed-index normal form}
\begin{proposition}\label{prop:nf}
\textup{(1)} Let $f:\Pone_k\to\Pone_k$ have degree $d\geq2$ and $d+1$ distinct fixed points.
After marking one fixed point as $x_0$ and conjugating it to $\infty$, write the remaining
fixed points as $x_1,\ldots,x_d\in k$. Write $c_i=\iota_f(x_i)$ for $0\leq i\leq d$. Then
\begin{equation}\label{eq:nf}
	f(z)=z-\frac{1}{R(z)},\quad R(z)=\sum_{i=1}^d\frac{c_i}{z-x_i},
\end{equation}
and $c_0=1-\sum_{i=1}^d c_i$.

\textup{(2)} Conversely, let $x_1,\ldots,x_d\in k$ be pairwise distinct, and let $c=(c_0,c_1,\ldots,c_d)$ be an ordered fixed-point index vector. Then \eqref{eq:nf} defines a degree-$d$ rational map with fixed points $x_0=\infty,x_1,\ldots,x_d$ and $c_i=\iota_f(x_i)$ for $0\leq i\leq d$.
\end{proposition}

\begin{proof}
\textup{(1)}. Clearly, every fixed point of $f$ is simple. Define
\[
R_f(z):=\frac{1}{z-f(z)}.
\]
Near $x_i$, one has
\[
z-f(z)=\left(1-\rho_f(x_i)\right)(z-x_i)+O((z-x_i)^2),
\]
so the residue of $R_f(z)\,dz$ at $x_i$ is $c_i$. These are the only finite poles of $R_f$. At a pole of $f$, the function $R_f$ has a zero, while at a finite point that is neither fixed nor a pole it is regular.

Since $\infty$ is a simple fixed point, either $f$ has local degree at least two at $\infty$, in which case $R_f(z)=O(z^{-2})$, or $f$ has local degree one there. In the latter case, $f(z)=az+O(1)$ with $a\neq1$, and $R_f(z)=O(z^{-1})$. Hence $R_f$ has no polynomial part. Its unique partial-fraction decomposition is the function $R$ in \eqref{eq:nf}. Solving $R_f=1/(z-f(z))$ gives the asserted formula. The fixed-point index formula \cite[Exercise~1.17]{Sil07} gives $c_0=1-\sum_{i=1}^d c_i$.

\smallskip

\textup{(2)}. For the converse, define
\[
Q(z):=\prod_{i=1}^d(z-x_i),\quad A(z):=\sum_{i=1}^d c_i\frac{Q(z)}{z-x_i}.
\]
Then
\[
f(z)=\frac{zA(z)-Q(z)}{A(z)}.
\]
Since $A(x_i)=c_iQ'(x_i)\neq0$, the numerator and denominator have no common zero. If $C=\sum_{i=1}^d c_i$, then the coefficient of $z^d$ in $zA-Q$ is $C-1=-c_0\neq0$. Hence $f$ has degree $d$. Its finite fixed-point equation is $Q(z)=0$, and $\infty$ is also fixed. Hence, the $d+1$ fixed points $x_0=\infty,x_1,\ldots,x_d$ of $f$ are distinct and therefore simple. Applying \textup{(1)} to $f$, we obtain that $c_i=\iota_f(x_i)$ for $0\leq i\leq d$.
\end{proof}

\subsection{A multivariable Hensel's lemma}
We use the following multivariable form of Hensel's lemma for restricted power series; see \cite[Ch.~III, \S4, no.~5]{BouCA}. We include the norm estimates needed below.
\begin{lemma}\label{lem:hensel}
Let $K$ be a complete field equipped with a nontrivial non-archimedean absolute value $|\cdot|$. We write $\sO_K$ for its valuation ring, $\fm_K$ for its maximal ideal, and $\kappa_K=\sO_K/\fm_K$ for its residue field. Let $r\geq1$, and let $F=(F_1,\ldots,F_r)$ be an $r$-tuple of strictly convergent power series in $\sO_K\langle X_1,\ldots,X_r\rangle$. Let $\overline{F}$ be its reduction in $\kappa_K[X_1,\ldots,X_r]$. If $\bar{u}\in\kappa_K^r$ satisfies
\[
\overline{F}(\bar{u})=0\quad\text{and}\quad\det J_{\overline{F}}(\bar{u})\neq0,
\]
then there is a unique $u\in\sO_K^r$ reducing to $\bar{u}$ such that $F(u)=0$. More generally, if $b\in\sO_K^r$, $\overline{F}(\bar{u})=\bar{b}$, and $\det J_F$ is a unit on the residue polydisc above $\bar{u}$, then $F(u)=b$ has a unique solution in that residue polydisc.
\end{lemma}
\begin{proof}
Choose a lift $u_0$ of $\bar{u}$. Then $F(u_0)\in\fm_K^r$ and $J_F(u_0)\in\mathrm{GL}_r(\sO_K)$. Newton iteration
\[
u_{m+1}=u_m-J_F(u_m)^{-1}F(u_m)
\]
is defined in the residue polydisc. With the sup norm, by \cite[Ch.~III, \S4, no.~5]{BouCA}, we obtain
\[
\lVert F(u_{m+1})\rVert\leq\lVert F(u_m)\rVert^2,\quad\lVert u_{m+1}-u_m\rVert\leq\lVert F(u_m)\rVert.
\]
Thus, the sequence $(u_m)$ is Cauchy, hence converges; denote its limit by $u$. Then $u$ is a zero of $F$ that reduces to $\bar{u}$.

If $u,u'$ are two zeros in the same residue polydisc with $\delta:=u-u'\neq0$, then $0<\lVert\delta\rVert<1$. Taylor expansion at $u'$ yields
\[
0=J_F(u')\delta+E,\quad\lVert E\rVert\leq\lVert\delta\rVert^2.
\]
Since $J_F(u')^{-1}$ has integral entries, this implies $\lVert\delta\rVert\leq\lVert\delta\rVert^2$, a contradiction.

Replacing $F$ by $F-b$ implies the last conclusion.
\end{proof}

\subsection{The pair-cycle expansion}
\begin{proposition}\label{prop:exp}
Let $d\geq2$, and let $K$ be a complete non-archimedean field of residue characteristic different from $2$. Take
\[
t\in\fm_K\setminus\{0\},\quad w_1,\ldots,w_d\in\sO_K^\times,\quad x\in\Conf_d(\A^1)(K).
\]
Define
\[
S_x(z):=\sum_{r=1}^d\frac{w_r}{z-x_r},\quad f_{x,t}(z):=z-\frac{1}{tS_x(z)}.
\]
Fix $i\neq j$. Suppose that $x_i-x_j$ is a unit and that
\[
(x_i-x_r)^{-1},\ (x_j-x_r)^{-1}\in\sO_K\quad(r\neq i,j).
\]
Then $f_{x,t}$ has a unique two-cycle with one point in $x_i+t\sO_K$ and the other in $x_j+t\sO_K$. If $p_{ij}$ is either point of this cycle, set
\[
\Delta:=x_i-x_j,\quad s_\ell:=\sum_{r\neq\ell}\frac{w_r}{x_\ell-x_r}\quad(\ell=i,j).
\]
Then
\begin{equation}\label{eq:rho}
\rho_{f_{x,t}}(p_{ij})=\frac{1}{t^2w_iw_j}-\frac{1}{t}\left(\frac{1}{w_i}+\frac{1}{w_j}+\frac{2\Delta(s_i-s_j)}{w_iw_j}\right)+O(1).
\end{equation}
Here $O(1)$ denotes an element of $\sO_K$. On an affinoid family on which the stated unit conditions remain valid, the $O(1)$ term is analytic in the configuration variables and remains $O(1)$ after differentiation with respect to those variables.
\end{proposition}
\begin{proof}
We seek $A,D\in\sO_K$ such that
\[
p=x_i+tA,\quad q=x_j+tD,\quad f_{x,t}(p)=q,\quad f_{x,t}(q)=p.
\]
Set $\Delta=x_i-x_j$, and define
\[
H_i(A):=\sum_{r\neq i}\frac{w_r}{x_i+tA-x_r},\quad H_j(D):=\sum_{r\neq j}\frac{w_r}{x_j+tD-x_r}.
\]
The hypotheses on the inverse differences imply that
\[
H_i\in\sO_K\langle A\rangle,\quad H_j\in\sO_K\langle D\rangle.
\]
Note that
\[
S_x(p)=S_x(x_i+tA)=\frac{w_i}{tA}+H_i(A),
\]
and then
\[
\frac{1}{tS_x(p)}=\frac{1}{tS_x(x_i+tA)}=\frac{A}{w_i+tAH_i(A)}.
\]
The analogous identity holds at $q$. Thus, the equations $f_{x,t}(p)=q$ and $f_{x,t}(q)=p$ are equivalent to
\[
F_1(A,D)=0,\quad F_2(A,D)=0,
\]
where
\[
F_1(A,D):=\Delta+t(A-D)-\frac{A}{w_i+tAH_i(A)}
\]
and
\[
F_2(A,D):=-\Delta+t(D-A)-\frac{D}{w_j+tDH_j(D)}.
\]

We can apply Lemma~\ref{lem:hensel} with
\[
r=2,\quad F=(F_1,F_2),\quad b=(0,0),\quad\text{and}\quad\bar{u}=\left(\bar{w}_i\bar{\Delta},-\bar{w}_j\bar{\Delta}\right)\in\kappa_K^2.
\]
Indeed, reduction modulo $\fm_K$ gives
\[
\overline{F_1}(\bar{A},\bar{D})=\bar\Delta-\frac{\bar{A}}{\bar{w}_i}\quad\text{and}\quad\overline{F_2}(\bar{A},\bar{D})=-\bar{\Delta}-\frac{\bar{D}}{\bar{w}_j}.
\]
The unique solution of the reduced system is therefore
\[
\left(\bar{A},\bar{D}\right)=\left(\bar{w}_i\bar{\Delta},-\bar{w}_j\bar{\Delta}\right)=\bar{u}.
\]
It is direct to compute that $\det J_{\overline{F}}(\bar{u})=1/(\bar{w}_i\bar{w}_j)\neq0$. Consequently, Lemma~\ref{lem:hensel} gives a unique solution
\[
(A,D)\in\sO_K^2
\]
with the indicated reduction. Since $x_i-x_j$ is a unit, the residue discs of $x_i$ and $x_j$ are disjoint. Hence $p\neq q$, and the resulting orbit has exact period two.

We next refine the congruence obtained by Lemma~\ref{lem:hensel}. Set
\[
u=(A,D),\quad u_0=(w_i\Delta,-w_j\Delta).
\]
Since
\[
\frac{w_i\Delta}{w_i+t w_i\Delta H_i(w_i\Delta)}=\frac{\Delta}{1+t\Delta H_i(w_i\Delta)}\equiv\Delta\pmod{t\sO_K},
\]
and similarly at $x_j$, one has
\[
F(u_0)\in t\sO_K^2.
\]
Also $u\equiv u_0\pmod{\fm_K}$. For each coordinate of $F$, write $F_\ell(u)-F_\ell(u_0)$ as a telescoping sum in the two variables, which yields a matrix $M(u,u_0)\in\mathrm{Mat}_{2\times2}(\sO_K)$ such that
\[
F(u)-F(u_0)=M(u,u_0)(u-u_0).
\]
Since $u\equiv u_0\pmod{\fm_K}$, reduction gives
\[
\overline{M(u,u_0)}=J_{\overline F}(\bar u)\in\mathrm{GL}_2(\kappa_K),
\]
and then $M(u,u_0)\in\mathrm{GL}_2(\sO_K)$. Since $F(u)=0$, it follows that
\[
u-u_0=-M(u,u_0)^{-1}F(u_0)\in t\sO_K^2.
\]
Thus, 
\[
A=w_i\Delta+O(t)\quad\text{and}\quad D=-w_j\Delta+O(t).
\]
For $m\in\Z_{\geq0}$, $O(t^m)$ denotes an element of $t^m\sO_K$. Consequently,
\begin{equation}\label{eq:pq}
p=x_i+tw_i\Delta+O(t^2)\quad\text{and}\quad q=x_j-tw_j\Delta+O(t^2).
\end{equation}

For $z$ near $x_i$, writing $y=z-x_i$, we compute that
\[
S_x(x_i+y)=\frac{w_i}{y}+s_i+O(y),
\]
and hence
\[
\frac{1}{S_x(x_i+y)}=\frac{1}{w_i}y-\frac{s_i}{w_i^2}y^2+O(y^3).
\]
Therefore,
\[
\left(\frac{1}{S_x}\right)'(x_i+y)=\frac{1}{w_i}-\frac{2s_i}{w_i^2}y+O(y^2).
\]
Since
\[
f_{x,t}'(z)=1-\frac{1}{t}\left(\frac{1}{S_x}\right)'(z),
\]
substitution of
\[
p-x_i=tw_i\Delta+O(t^2)
\]
gives
\[
\begin{aligned}
f_{x,t}'(p)&=1-\frac{1}{tw_i}+\frac{2s_i}{tw_i^2}\left(tw_i\Delta+O(t^2)\right)+O(t)\\
&=-\frac{1}{tw_i}+1+\frac{2\Delta s_i}{w_i}+O(t).
\end{aligned}
\]
Likewise, we have
\[
f_{x,t}'(q)=-\frac{1}{tw_j}+1-\frac{2\Delta s_j}{w_j}+O(t).
\]
The product of $f_{x,t}'(p)$ and $f_{x,t}'(q)$ shows \eqref{eq:rho}. The valuation of $\rho_{f_{x,t}}(p)=\rho_{f_{x,t}}(q)$ is $-2v(t)<0$, so in particular this multiplier is not $1$. Thus, both $p$ and $q$ are simple fixed points of $f_{x,t}^{\circ2}$.

Finally, consider an affinoid family of configurations on which the hypotheses of the proposition hold uniformly. The functions $H_i,H_j$, and hence $F_1,F_2$, then belong to the corresponding affinoid Tate algebra in the configuration variables and $A,D$. The Jacobian with respect to $(A,D)$ remains a unit on the residue polydisc. The non-archimedean analytic implicit-function theorem therefore yields analytic functions $A$ and $D$ of the configuration variables. The above argument works in the affinoid algebra as well and shows
\[
A=w_i\Delta+t\alpha,\quad D=-w_j\Delta+t\delta
\]
with bounded analytic $\alpha,\delta$. Substitution into the convergent local expansions implies that the remainder in \eqref{eq:rho} is a bounded analytic function. Differentiating the implicit equations, one has
\[
\partial_a
\begin{pmatrix}A\\ D\end{pmatrix}
=-\left(J_{(A,D)}F\right)^{-1}\,\partial_a F
\]
for any local configuration coordinate $a$, and these derivatives are again bounded. It follows that the $O(1)$ term in \eqref{eq:rho} remains $O(1)$ after differentiation with respect to the configuration variables.
\end{proof}
The coefficient of $t^{-1}$ in \eqref{eq:rho} is essential: the first term $1/(t^2w_iw_j)$ determines only the pair label, whereas the next term contains $B_{ij}(x)$. This distinction is used in both the reconstruction and the boundary argument.

\section{Pair invariants and boundary control}\label{sec:pair}
\subsection{Factorization of difference products}
For $u,v\in\A^d(k)$, we use the symmetric matrix $M(u,v)$ defined in \S~\ref{sec:mainidea}, whose $(i,j)$-entry is $(u_i-u_j)(v_i-v_j)$.

\begin{proposition}\label{prop:fac}
Let $d\geq2$ be an integer. Then there is a Zariski dense open subset $U$ of $\A^d(k)\times\A^d(k)$ such that, for $(u,v)\in U$, if $(u',v')\in\A^d(k)\times\A^d(k)$ satisfies $M(u',v')=M(u,v)$, then
\[
(u',v')=(au+b,a^{-1}v+c)\quad\text{or}\quad(u',v')=(a^{-1}v+c,au+b)
\]
for some $(a,b,c)\in k^\times\times k^2$.
\end{proposition}
\begin{proof}
Set $V=k^2$ and $q(x,y)=xy$. The symmetric bilinear form associated with
$q$ is
\[
B_q\left((x,y),(x',y')\right):=q\left((x,y)+(x',y')\right)-q(x,y)-q(x',y')=xy'+x'y,
\]
which is nondegenerate. So $(V,q)$ is a regular quadratic space. For $1\leq i\leq d$, define $p_i=(u_i,v_i),\ p_i':=(u_i',v_i')\in V$. Then
\[
M(u,v)_{ij}=q(p_i-p_j)\quad\text{and}\quad M(u',v')_{ij}=q(p_i'-p_j')\quad(1\leq i,j\leq d).
\]
Then $M(u',v')=M(u,v)$ is equivalent to
\begin{equation}\label{eq:qdist}
q(p_i'-p_j')=q(p_i-p_j)\quad(1\leq i,j\leq d).
\end{equation}

\smallskip

Suppose first that $d\geq3$. Define
\[
U:=\left\{(u,v)\in\A^d(k)\times\A^d(k)\colon\det
\begin{pmatrix}
	u_2-u_1&v_2-v_1\\
	u_3-u_1&v_3-v_1
\end{pmatrix}
\neq0\right\},
\]
which is Zariski open dense in $\A^d(k)\times\A^d(k)$. Fix $(u,v)\in U$, and suppose that $(u',v')$ satisfies \eqref{eq:qdist}. For $1\leq i\leq d$, define $e_i=p_i-p_1,\ e_i'=p_i'-p_1'\in V$. By the parallelogram identity and \eqref{eq:qdist}, for $1\leq i,j\leq d$ we have
\begin{equation}\label{eq:gram}
B_q(e_i',e_j')=q(e_i')+q(e_j')-q(e_i'-e_j')=q(e_i)+q(e_j)-q(e_i-e_j)=B_q(e_i,e_j).
\end{equation}
The definition of $U$ shows that $e_2,e_3$ form a $k$-basis of $V$. Let $\sigma:V\to V$ be the $k$-linear map determined by $\sigma(e_2)=e_2'$ and $\sigma(e_3)=e_3'$. From \eqref{eq:gram}, we conclude that
\[
B_q(\sigma(x),\sigma(y))=B_q(x,y)\quad(x,y\in V).
\]
In particular, $\sigma$ is injective, and $\sigma\in O(q)(k)$ (i.e., $\sigma$ is an isometry of $(V,q)$). Then $e_2',e_3'$ also form a basis of $V$. For $i\geq4$ and $j\in\{2,3\}$, by \eqref{eq:gram} we obtain
\[
B_q\left(e_i'-\sigma(e_i),e_j'\right)=B_q(e_i',e_j')-B_q\left(\sigma(e_i),\sigma(e_j)\right)=B_q(e_i,e_j)-B_q(e_i,e_j)=0.
\]
Since $e_2',e_3'$ form a basis of $V$ and $B_q$ is nondegenerate, we see that $e_i'=\sigma(e_i)$ for $i\geq4$. Clearly, $e_i'=\sigma(e_i)$ also holds for $1\leq i\leq3$. Thus,
\begin{equation}\label{eq:affiso}
p_i'=p_1'+\sigma(p_i-p_1)\quad(1\leq i\leq d).
\end{equation}

\smallskip

Suppose now that $d=2$. Define
\[
U:=\left\{(u,v)\in\A^2(k)\times\A^2(k)\colon M(u,v)_{12}\neq0\right\}.
\]
Clearly, $U$ is Zariski open dense. Fix $(u,v)\in U$. Set $(e,e')=(p_2-p_1,p_2'-p_1')$. Equation~\eqref{eq:qdist} shows $q(e')=q(e)=M(u,v)_{12}\neq0$. Then $B_q(e,e)=2q(e)\neq0$, so the one-dimensional quadratic subspace $(W:=e\cdot k,q\vert_W)$ is regular. The $k$-linear map
\[
\sigma_0:W\to e'\cdot k,\quad\sigma_0(e)=e',
\]
is an isometry with respect to the restrictions of $q$. By Witt's extension theorem (cf.~\cite[p.~15]{Sch85}), $\sigma_0$ extends to an isometry $\sigma\in O(q)(k)$. Observe that \eqref{eq:affiso} also holds when $d=2$.

It remains to determine the elements of $O(q)(k)$. Write
\[
\sigma(x,y)=(\alpha x+\beta y,\gamma x+\delta y),
\]
where $\alpha,\beta,\gamma,\delta\in k$ satisfy $\alpha\delta-\beta\gamma\neq0$. From $\sigma\in O(q)(k)$, we obtain
\[
\alpha\gamma=\beta\delta=0\quad\text{and}\quad\alpha\delta+\beta\gamma=1;
\]
then either
\[
\beta=\gamma=0,\quad \alpha\delta=1,
\]
or
\[
\alpha=\delta=0,\quad \beta\gamma=1.
\]
Therefore, for some $a\in k^\times$, either $\sigma(x,y)=(ax,a^{-1}y)$ or $\sigma(x,y)=(a^{-1}y,ax)$. In the first case, write $p_1'-\sigma(p_1)=(b,c)$. Then \eqref{eq:affiso} shows $(u',v')=(au+b,a^{-1}v+c)$. In the second case, write $p_1'-\sigma(p_1)=(c,b)$. Then \eqref{eq:affiso} implies $(u',v')=(a^{-1}v+c,au+b)$.
\end{proof}

\subsection{Reconstruction from the pair invariants}
For $d\geq2$, the group $\Aff$ acts freely on $\Conf_d(\A^1)$. The normalization $(x_1,x_2)=(0,1)$ identifies $\sC_d=\Conf_d(\A^1)/\Aff$ with
\[
\left\lbrace(x_3,\ldots,x_d)\in\A^{d-2}\colon 0,1,x_3,\ldots,x_d\text{ are pairwise distinct}\right\rbrace.
\]
In particular, $\sC_d$ is irreducible of dimension $d-2$.

Let
\[
w=(w_1,\ldots,w_d)\in k^d\quad\text{and}\quad x=(x_1,\ldots,x_d)\in\Conf_d(\A^1)(k).
\]
Define
\begin{equation}\label{eq:B}
T_w(x)_i=\sum_{r\neq i}\frac{w_r}{x_i-x_r},\quad B_{ij}(x)=(x_i-x_j)\left(T_w(x)_i-T_w(x)_j\right)=M\left(x,T_w(x)\right)_{ij}.
\end{equation}
For $(a,b)\in k^\times\times k$, we have 
\[
T_w(ax+b)=a^{-1}T_w(x)\quad\text{and}\quad B_{ij}(ax+b)=B_{ij}(x).
\]

Here and below, a statement for ``generic weights'' $w$ means that it holds outside a proper Zariski closed subset of the weight space $\A^d(k)$. For a fixed weight $w\in\A^d(k)$, the functions in \eqref{eq:B} induce a morphism
\[
\beta_w:\sC_{d,k}\to\A^{\binom{d}{2}}_k,\quad[x]\mapsto(B_{ij}(x))_{i<j}.
\]
Define
\[
Y_w:=\overline{\beta_w(\sC_d)}\subseteq\A^{\binom{d}{2}}_k.
\]
Then the morphism $\beta_w:\sC_{d,k}\to Y_w$ is dominant.

\begin{proposition}\label{prop:B}
Let $d\geq2$ be an integer. For $w$ outside a proper Zariski closed subset of $\A^d(k)$, the morphism $\beta_w:\sC_{d,k}\to Y_w$ is birational, and its differential has rank $d-2$ at the generic point of $\sC_{d,k}$.
\end{proposition}
\begin{proof}
If $d=2$, then $\sC_2$ consists of a single point, and the conclusion is trivial. Suppose now that $d\geq3$.

Use the normalization $(x_1,x_2)=(0,1)$ and view $\sC_d\subseteq\A^{d-2}$. Define
\[
\Delta_w(x):=\det
\begin{pmatrix}
	x_2-x_1&T_w(x)_2-T_w(x)_1\\
	x_3-x_1&T_w(x)_3-T_w(x)_1
\end{pmatrix}.
\]
When $T_w(x)\in\Conf_d(\A^1)$, define
\[
T_w^2(x):=T_w\left(T_w(x)\right)\quad\text{and}\quad G_w(x):=\left(T_w^2(x)_2-T_w^2(x)_1\right)-(x_2-x_1).
\]

\smallskip

We will verify that, with $x_3,\ldots,x_d,w_1,\ldots,w_d$ regarded as independent variables, the rational functions $\Delta_w(x)$ and $G_w(x)$ are nonzero and that $T_w(x)$ is generically a configuration. The verification below is uniform in every characteristic different from $2$.

For $d=3$, write
\[
(x_1,x_2,x_3)=(0,1,z)\quad\text{and}\quad(w_1,w_2,w_3)=(a,b,c),
\]
where $a,b,c,z$ are independent variables. A direct calculation gives
\[
T_w(x)_1=-\frac{bz+c}{z},\quad
T_w(x)_2=\frac{az-a-c}{z-1},\quad
T_w(x)_3=\frac{(a+b)z-a}{z(z-1)}.
\]
Then
\[
\Delta_w(x)=-\frac{P}{z(z-1)}
\]
and
\[
G_w(x)=
\frac{c(z^2-z+1)\left((a+b)z-a\right)P}
{\left(bz^2+(a+c)z-a-c\right)\left(az^2-(2a+b+c)z+a\right)\left((a+b)z^2-(a+b)z-c\right)},
\]
where
\[
P=(a+b)z^3-(a+2b)z^2-(a+2c)z+a+c.
\]
Both $\Delta_w(x)$ and $G_w(x)$ are nonzero rational functions. Moreover, using the common denominator $z(z-1)$, the three pairwise differences among the coordinates of $T_w(x)$ have numerators
\[
-(a+b)z^2+(a+b)z+c,\quad
-bz^2-(a+c)z+a+c,\quad
az^2-(2a+b+c)z+a,
\]
which are nonzero polynomials. Since $k$ is infinite, we conclude that $T_w(x)$ is generically a configuration.

For $d\geq4$, specialize
\[
(w_1,w_2,w_3,w_4,\ldots,w_d)=(a,b,c,0,\ldots,0)
\]
and keep $(x_1,x_2,x_3)=(0,1,z)$. The first three coordinates of $T_w(x)$ and of $T_w^2(x)$ remain those in the preceding $d=3$ computation. For $r\geq4$,
\[
T_w(x)_r=\frac{a}{x_r}+\frac{b}{x_r-1}+\frac{c}{x_r-z},
\]
which is a nonconstant rational function of $x_r$. Taking $x_4,\ldots,x_d$ successively as independent variables therefore shows that the coordinates of $T_w(x)$ are generically pairwise distinct. Consequently, the same specialization argument shows that $\Delta_w(x)$ and $G_w(x)$ are nonzero rational functions for every $d\geq4$.

It follows that
\[
\Delta_w(x),\quad G_w(x),\quad\prod_{i<j}\left(T_w(x)_i-T_w(x)_j\right)
\]
are nonzero rational functions on the total $(w,[x])$-parameter space $\A^d\times\sC_d$. After clearing denominators, the condition that any one of these rational functions vanish identically as a function of $x_3,\ldots,x_d$ is a Zariski closed condition on $w$. The preceding computations show that the union of these closed conditions is proper. Hence there is a Zariski open dense subset $W^\circ\subseteq\A^d(k)$ such that, for every $w\in W^\circ$, there is a Zariski open dense subset $\Omega_w\subseteq\sC_d(k)$ on which
\[
\Delta_w(x)\neq0,\quad T_w(x)\in\Conf_d(\A^1)(k),\quad G_w(x)\neq0.
\]

\smallskip

Fix $w\in W^\circ$ and $[x]\in\Omega_w$. Suppose that $[x']\in\sC_d(k)$ satisfies $\beta_w([x'])=\beta_w([x])$. Then
\[
M\left(x',T_w(x')\right)=M\left(x,T_w(x)\right).
\]
Since $\Delta_w(x)\neq0$, by the proof of Proposition~\ref{prop:fac}, there exist $a\in k^\times$ and $b,c\in k$ such that
\[
(x',T_w(x'))=(ax+b,a^{-1}T_w(x)+c)\quad\text{or}\quad(x',T_w(x'))=(a^{-1}T_w(x)+c,ax+b).
\]
Suppose first that $(x',T_w(x'))=(a^{-1}T_w(x)+c,ax+b)$. We obtain
\[
T_w(x')=T_w\left(a^{-1}T_w(x)+c\right)=a\,T_w^2(x).
\]
Comparison with $T_w(x')=ax+b$ yields $T_w^2(x)=x+b/a$. Then
\[
T_w^2(x)_2-T_w^2(x)_1=x_2-x_1\quad\text{and}\quad G_w(x)=0,
\]
contradicting $[x]\in\Omega_w$. Thus, we must have $(x',T_w(x'))=(ax+b,a^{-1}T_w(x)+c)$. Note that $[x']=[ax+b]=[x]$ in $\sC_d(k)$. Thus,
\begin{equation}\label{eq:fibersing}
\beta_w^{-1}\left(\beta_w([x])\right)=\left\lbrace[x]\right\rbrace.
\end{equation}

By \eqref{eq:fibersing}, the morphism $\beta_w:\sC_{d,k}\to Y_w$ is generically finite. We next show that the induced function-field extension is separable. Continue to use the normalization $(x_1,x_2)=(0,1)$, and set
\[
t_i:=T_w(x)_i-T_w(x)_1.
\]
Then, for $i\geq3$,
\[
B_{12}=t_2,\quad B_{1i}=x_it_i,\quad B_{2i}=(1-x_i)(t_2-t_i).
\]
Hence, with
\[
C_i:=B_{12}+B_{1i}-B_{2i}=t_i+x_it_2,
\]
the element $x_i$ satisfies
\begin{equation}\label{eq:sepquad}
B_{12}X^2-C_iX+B_{1i}=0,
\end{equation}
where the coefficients are all in $k(Y_w)$. The derivative of the left-hand side of \eqref{eq:sepquad} at $X=x_i$ is
\[
2B_{12}x_i-C_i=x_it_2-t_i=-\Delta_{w,i}(x),
\]
where
\[
\Delta_{w,i}(x):=\det
\begin{pmatrix}
	x_2-x_1&t_2\\
	x_i-x_1&t_i
\end{pmatrix}.
\]
Note that $\Delta_{w,3}(x)=\Delta_w(x)$ is a nonzero rational function. For every $i\geq3$, the rational function $\Delta_{w,i}$ is nonzero by interchanging labels $3$ and $i$ in the calculation above. It follows that each $x_i$ is separable over $k(Y_w)$ at the generic point. Since $k(\sC_d)=k(x_3,\ldots,x_d)$, the finite extension $k(\sC_d)/k(Y_w)$ is separable. Combining separability with \eqref{eq:fibersing}, we obtain $[k(\sC_d):k(Y_w)]=1$. Thus, $\beta_w:\sC_{d,k}\to Y_w$ is birational.

Finally, a birational morphism between irreducible varieties restricts to an isomorphism between Zariski open dense subsets. Since $\dim\sC_d=d-2$, the differential of $\beta_w$ has rank $d-2$ at the generic point of $\sC_{d,k}$.
\end{proof}

\subsection{The cluster identity}
\begin{lemma}\label{lem:cl}
Let $d\geq2$ be an integer, and let $w_1,\ldots,w_d\in F$, where $F$ is a field. Suppose that $C\subseteq\{1,\ldots,d\}$ has cardinality $\geq2$. Let $(x_i)_{i\in C}$ be pairwise distinct elements of $F$. Define
\begin{align*}
s_i^C&=s_i^C(x)=\sum_{r\in C\setminus\{i\}}\frac{w_r}{x_i-x_r},\quad B_{ij}^C=B_{ij}^C(x)=(x_i-x_j)(s_i^C-s_j^C),\\
w_C&=(w_i)_{i\in C},\quad W_C=\sum_{i\in C}w_i,
\quad e_2(w_C)=\sum_{\substack{i<j\\i,j\in C}}w_iw_j.
\end{align*}

\textup{(1)} We have
\begin{equation}\label{eq:cl}
\sum_{\substack{i<j\\i,j\in C}}w_iw_jB_{ij}^C=W_Ce_2(w_C).
\end{equation}

\textup{(2)} Regard $w_1,\ldots,w_d$ and $x_1,\ldots,x_d$ as independent variables over a field. If $C$ is proper, then
\begin{equation}\label{eq:clg}
\sum_{\substack{i<j\\i,j\in C}}w_iw_jB_{ij}(x)-W_Ce_2(w_C)=-\sum_{r\notin C}w_r\sum_{\substack{i<j\\i,j\in C}}w_iw_j\frac{(x_i-x_j)^2}{(x_i-x_r)(x_j-x_r)}.
\end{equation}
For every proper $C$, the right-hand side of \eqref{eq:clg} is a nonzero rational function in $w_1,\ldots,w_d,x_1,\ldots,x_d$.
\end{lemma}
\begin{proof}
\textup{(1)}. Note that
\begin{equation}\label{eq:sum-s}
\sum_{i\in C}w_i s_i^C=\sum_{\substack{i,r\in C\\ i< r}}\left(\frac{w_iw_r}{x_i-x_r}+\frac{w_rw_i}{x_r-x_i}\right)=0.
\end{equation}
Similarly,
\begin{equation}\label{eq:sum-xs}
\sum_{i\in C}w_i x_i s_i^C=\sum_{\substack{i<r\\ i,r\in C}}w_iw_r\left(\frac{x_i}{x_i-x_r}+\frac{x_r}{x_r-x_i}\right)=\sum_{\substack{i<r\\ i,r\in C}}w_iw_r=e_2(w_C).
\end{equation}

For arbitrary families $(a_i)_{i\in C}$ and $(b_i)_{i\in C}$, direct
expansion gives
\begin{equation}\label{eq:weighted-difference}
\sum_{\substack{i<j\\ i,j\in C}}w_iw_j(a_i-a_j)(b_i-b_j)=W_C\sum_{i\in C}w_i a_i b_i-\left(\sum_{i\in C}w_i a_i\right)\left(\sum_{i\in C}w_i b_i\right).
\end{equation}
Applying \eqref{eq:weighted-difference} with $a_i=x_i$ and $b_i=s_i^C$, and then using \eqref{eq:sum-s} and \eqref{eq:sum-xs}, we obtain
\[
\sum_{\substack{i<j\\ i,j\in C}}w_iw_jB_{ij}^C=W_C\sum_{i\in C}w_i x_i s_i^C-\left(\sum_{i\in C}w_i x_i\right)\left(\sum_{i\in C}w_i s_i^C\right)=W_Ce_2(w_C).
\]

\smallskip

\textup{(2)}. Write
\[
T_w(x)_i=s_i^C+r_i\quad\text{and}\quad r_i=\sum_{r\notin C}\frac{w_r}{x_i-x_r}.
\]
Then, for $i,j\in C$,
\[
B_{ij}(x)=B_{ij}^C+(x_i-x_j)(r_i-r_j).
\]
For each $r\notin C$, we have
\[
(x_i-x_j)\left(\frac{1}{x_i-x_r}-\frac{1}{x_j-x_r}\right)=-\frac{(x_i-x_j)^2}{(x_i-x_r)(x_j-x_r)}.
\]
Consequently,
\[
B_{ij}(x)-B_{ij}^C=-\sum_{r\notin C}w_r\frac{(x_i-x_j)^2}{(x_i-x_r)(x_j-x_r)}.
\]
Then \eqref{eq:clg} follows by multiplying by $w_iw_j$, summing over $i<j$ with $i,j\in C$, and \textup{(1)}.

Finally, suppose that $C\subsetneq\{1,\ldots,d\}$ with $\# C\geq2$. Fix $r\notin C$ and distinct $i,j\in C$. In the right-hand side of \eqref{eq:clg}, the monomial $w_rw_iw_j$ occurs only in the summand indexed by $(r,i,j)$, and its coefficient is
\[
-\frac{(x_i-x_j)^2}{(x_i-x_r)(x_j-x_r)}\neq0.
\]
Therefore, the right-hand side of \eqref{eq:clg} is a nonzero rational function.
\end{proof}

\subsection{No escape to the boundary}
Let $d\geq2$ be an integer. Set $(x_1,x_2)=(0,1)$. Suppose that $w_1,\ldots,w_d,x_3,\ldots,x_d$ are algebraically independent over $k$. When $d=2$, only consider $w_1,\ldots,w_d$. Let
\[
L=k(w_1,\ldots,w_d,x_3,\ldots,x_d)((t))
\]
be equipped with the $t$-adic valuation $v$, normalized by $v(t)=1$. We use the same normalization on finite valued extensions of $L$.

\begin{proposition}\label{prop:noescape}
With the notation above, define
\[
S_x(z)=\sum_{i=1}^d\frac{w_i}{z-x_i}\quad\text{and}\quad f_{x,t}(z)=z-\frac{1}{tS_x(z)}.
\]
Let $F/L$ be a finite extension of valued fields. Let $y=(y_1,\ldots,y_d)\in\Conf_d(\A^1)(F)$ be a second ordered $d$-configuration. Define
\[
S_y(z)=\sum_{i=1}^d\frac{w_i}{z-y_i}\quad\text{and}\quad f_{y,t}(z)=z-\frac{1}{tS_y(z)}.
\]
Assume that $f_{y,t}^{\circ2}$ has only simple fixed points and that
\begin{equation}\label{eq:exact2multeq}
\left\lbrace \rho_{f_{x,t}}(z_0)\colon z_0\in\Fix(f_{x,t}^{\circ2})\setminus\Fix(f_{x,t})\right\rbrace=\left\lbrace \rho_{f_{y,t}}(z_0)\colon z_0\in\Fix(f_{y,t}^{\circ2})\setminus\Fix(f_{y,t})\right\rbrace
\end{equation}
as multisets. Then, after a conjugacy $\sigma\in\Aff(F)$ if necessary, all $y_i$ lie in the valuation ring of $F$ and their reductions are pairwise distinct.
\end{proposition}
\begin{proof}
Since $y\in\Conf_d(\A^1)(F)$, its coordinates are pairwise distinct. If $d=2$, take $a=(y_2-y_1)^{-1}$ and $b=-y_1 a$. Then $ay_1+b=0$ and $ay_2+b=1$, and the conclusion is immediate. Suppose now that $d\geq3$.

Let $m:=\min_{i<j}v(y_j-y_i)<\infty$, and take $i_0<j_0$ with $v(y_{j_0}-y_{i_0})=m$. After conjugacy, we may replace every
$y_i$ by
\[
\frac{y_i-y_{i_0}}{y_{j_0}-y_{i_0}}.
\]
Then $y_i\in\sO_F$ for all $i$, while $y_{i_0}=0$ and $y_{j_0}=1$.

Suppose, for contradiction, that the reductions $(\bar{y}_i)_i$ are not pairwise distinct. In particular, for some indices $i\neq j$, we have $v(y_i-y_j)>0$. Set
\[
q:=\max\left\lbrace v(y_i-y_j)\colon i\neq j\right\rbrace.
\]
Then $0<q<\infty$. Define
\[
i\sim j\quad\Longleftrightarrow\quad v(y_i-y_j)\geq q,
\]
which is clearly an equivalence relation. Take an equivalence class $C\subseteq\{1,\ldots,d\}$ with $\# C\geq2$, and fix two different indices $i_1,j_1\in C$. Note that $C\neq\{1,\ldots,d\}$ because $v(y_{j_0}-y_{i_0})=v(1)=0$. Set $\rho:=y_{j_1}-y_{i_1}\neq0$. Then the maximality of $q$ implies $v(\rho)=q$. For $1\leq r\leq d$, define $\xi_r:=\left(y_r-y_{i_1}\right)/\rho$. Observe that $(\xi_r)_{r\in C}\subseteq\sO_F$, and they have pairwise distinct reductions. If $r\notin C$, then
\[
v(\xi_r)<0\quad\text{and}\quad(\xi_i-\xi_r)^{-1}\in\fm_F\quad(i\in C).
\]
Since $f_{y,t}$ and $f_{\xi,t}$ are affinely conjugate via $z\mapsto\rho z+ y_{i_1}$, they have the same multiplier spectrum.

By Proposition~\ref{prop:nf}, the fixed points of $f_{x,t}$ are $x_0=\infty,x_1=0,x_2=1,x_3,\ldots,x_d$, and we have
\[
\rho_{f_{x,t}}(x_i)=1-\frac{1}{tw_i}\quad (1\leq i\leq d),\quad \rho_{f_{x,t}}(\infty)=-\frac{tW}{1-tW},
\]
where $W=\sum_{i=1}^d w_i$. In particular, all $d+1$ fixed points of $f_{x,t}$ are simple fixed points of $f_{x,t}^{\circ2}$. Proposition~\ref{prop:exp} gives, for every $1\leq i<j\leq d$, a two-cycle $O_{ij}$ whose multiplier has valuation $-2$. Further, these two-cycles are distinct. From
\[
(d+1)+2\binom d2=d^2+1=N_{d,2},
\]
we see that every two-cycle of $f_{x,t}$ must be $O_{ij}$ for some $1\leq i<j\leq d$. Moreover, Proposition~\ref{prop:exp} shows
\begin{equation}\label{eq:t2rhored}
\overline{t^2\rho_{f_{x,t}}(p_{ij})}=\frac{1}{w_iw_j},\quad p_{ij}\in O_{ij},\quad 1\leq i<j\leq d.
\end{equation}
Note that the elements $1/(w_iw_j)$ ($1\leq i<j\leq d$) are pairwise distinct.

Let $i<j$ be in $C$. Proposition~\ref{prop:exp} applies to the configuration $\xi$ for $(i,j)$: all differences between coordinates in $C$ are units, while the inverse differences involving exactly one index outside $C$ are integral. Thus, we get a two-cycle $O_{ij}^\xi$ of $f_{\xi,t}$ associated with $(i,j)$. By \eqref{eq:exact2multeq}, \eqref{eq:t2rhored}, and the distinct values $1/(w_rw_s)$, we conclude that
\[
\rho_{f_{\xi,t}}(p_{ij}^\xi)=\rho_{f_{x,t}}(p_{ij}),
\]
where $p_{ij}^\xi\in O_{ij}^\xi$ and $p_{ij}\in O_{ij}$. Since $2w_iw_j$ is a unit in $\sO_F$, comparing the coefficient of $t^{-1}$ in \eqref{eq:rho} implies
\[
B_{ij}(x)\equiv B_{ij}(\xi)\pmod{\fm_F}\quad(i<j,\ i,j\in C).
\]
All terms involving $r\notin C$ in $B_{ij}(\xi)$ belong to $\fm_F$ because $(\xi_i-\xi_r)^{-1},(\xi_j-\xi_r)^{-1}\in\fm_F$. Thus, we deduce that
\[
\overline{B_{ij}(x)}=\overline{B_{ij}^C(\xi)},
\]
where $B_{ij}^C(\xi)$ is as in Lemma~\ref{lem:cl}. By Lemma~\ref{lem:cl} and the fact that 
\[
k(x_3,\ldots,x_d,w_1,\ldots,w_d)\cap\fm_F=\{0\},
\]
we deduce that
\begin{equation}\label{eq:wBred}
-\sum_{r\notin C}w_r
\sum_{\substack{i<j\\i,j\in C}}w_iw_j
\frac{(x_i-x_j)^2}{(x_i-x_r)(x_j-x_r)}
=\sum_{\substack{i<j\\i,j\in C}}w_iw_jB_{ij}(x)-W_Ce_2(w_C)=0.
\end{equation}
However, since $C$ is proper, \eqref{eq:wBred} contradicts the final assertion of Lemma~\ref{lem:cl}\textup{(2)}. Therefore, the reductions of $y_1,\ldots,y_d$ are pairwise distinct.
\end{proof}

\section{Reconstruction}\label{sec:rec}
\subsection{The ordered fixed-point cover}
Let $d\geq2$ be an integer. Define
\[
\sI_d^{\circ}:=\left\{(c_0,c_1,\ldots,c_d)\in\G_{m,k}^{d+1}\colon\sum_{i=0}^d c_i=1,\ c_j\neq\frac{1}{2}\text{ for all }j\right\}\quad\text{and}\quad\sX_d^{\mathrm{ord}}:=\sI_d^{\circ}\times\sC_{d,k},
\]
which are irreducible $k$-varieties of dimension $d$ and $2d-2$, respectively. We always use the normalization $(x_1,x_2)=(0,1)$ for $[x]\in\sC_d$. For every $\alpha=(c,[x])\in\sX_d^{\mathrm{ord}}(k)$, Proposition~\ref{prop:nf} gives a conjugacy class
\[
\nu_d(\alpha):=[g_{\alpha}]\in\sM_{d,k}(k),\quad g_{\alpha}(z)=z-\left(\sum_{i=1}^d\frac{c_i}{z-x_i}\right)^{-1}.
\]
Using Proposition~\ref{prop:nf}, it is direct to verify that the above formula defines a dominant morphism
\[
\nu_d:\sX_d^{\mathrm{ord}}\to\sM_{d,k}.
\]
Since $\dim\sX_d^{\mathrm{ord}}=2d-2=\dim\sM_{d,k}$, the morphism $\nu_d$ is generically finite.

\smallskip

We next define the two-cycle multiplier polynomial on $\sX_d^{\mathrm{ord}}$. On the Zariski open dense locus in $\sM_d$ where no fixed point of $f$ has multiplier $\pm1$, the $2$-multiplier polynomial for $f$ (see \eqref{eq:mulpoly}) factors as
\begin{equation}\label{eq:split}
\Phi_{2,f}(T)=\prod_{z\in\Fix(f)}\left(T-\rho_f(z)^2\right)\prod_{O\colon\,2\text{-cycle of }f}\left(T-\rho_f(z_O)\right)^2,\quad z_O\in O,
\end{equation}
where two-cycles are counted with multiplicity.

Let $A=\sO(\sX_d^{\mathrm{ord}})$. For $\alpha=(c,[x])\in\sX_d^{\mathrm{ord}}(k)$, we have
\[
\Fix(g_\alpha)=\{x_0=\infty,x_1=0,x_2=1,x_3,\ldots,x_d\}
\]
and
\[
\rho_{g_\alpha}(x_i)=1-\frac{1}{c_i}\neq\pm1\quad(0\leq i\leq d).
\]
Define
\[
P_1(T):=\prod_{i=0}^d\left(T-\left(1-\frac{1}{c_i}\right)^2\right)\in A[T].
\]
Then \eqref{eq:split} gives
\[
\Phi_{2,g_\alpha}(T)=P_1(T)(\alpha)\cdot E_2(T)(\alpha)^2,\quad E_2(T)(\alpha):=\prod_{O\colon\,2\text{-cycle of }g_\alpha}\left(T-\rho_{g_\alpha}(z_O)\right),\quad z_O\in O.
\]
It is direct to verify that $E_2(T)\in A[T]$. For every $c\in\sI_d^{\circ}(k)$, the non-leading coefficients of $E_2(T)$ define a morphism
\[
\sigma_c:\sC_{d,k}\to\A^{\binom{d}{2}}_k.
\]
Clearly, $\left(\tau_{d,2}\circ\nu_d\right)\left(c,[x]\right)$ determines $\sigma_c([x])\in\A^{\binom{d}{2}}(k)$, for every $[x]\in\sC_d(k)$. Define
\[
\Sigma_d:\sX_d^{\mathrm{ord}}\to\sI_d^{\circ}\times\A_k^{\binom{d}{2}},\quad(c,[x])\mapsto\left(c,\sigma_c\left([x]\right)\right).
\]

\begin{theorem}\label{thm:ordered}
Let $d\geq2$ be an integer. Then the induced dominant morphism
\[
\Sigma_d:\sX_d^{\mathrm{ord}}\to\overline{\Sigma_d\left(\sX_d^{\mathrm{ord}}\right)}
\]
is birational. Thus, for a generic ordered fixed-point index vector $c$, $\sigma_c([x])$ generically determines the class $[x]$.
\end{theorem}
\begin{proof}
If $d=2$, then $\sC_2$ consists of one point. Hence $\Sigma_2$ identifies $\sX_2^{\mathrm{ord}}=\sI_2^{\circ}\times\sC_2$ with the graph of the morphism $c\mapsto\sigma_c([0,1])$, and the conclusion follows. Suppose now that $d\geq3$.

Let $\eta$ be the generic point of $\sI_d^{\circ}$. Use coordinates $c_1,\ldots,c_d$ with $c_0=1-\sum_i c_i$, and set
\[
E:=k(c_1,\ldots,c_d)=k(\eta),\quad K_0:=k(w_1,\ldots,w_d)((t)).
\]
The assignment $c_i\mapsto tw_i$ defines an embedding $E\hookrightarrow K_0$. With $(x_1,x_2)=(0,1)$, take $x_3,\ldots,x_d$ algebraically independent over $k(w_1,\ldots,w_d)$, and set
\[
L:=k(w_1,\ldots,w_d,x_3,\ldots,x_d)((t)).
\]

\smallskip

Let
\[
\varphi:X:=\sC_{d,K_0}\to\A_{K_0}^{N},\quad N=\binom{d}{2},
\]
be the base change of $\sigma_\eta$ along $E\hookrightarrow K_0$. We will show that $\varphi$ is generically finite.

For $a=(a_1,\ldots,a_N)\in\A_{K_0}^{N}$, write
\[
Q_a(T):=T^N+a_1T^{N-1}+\cdots+a_N.
\]
Let $U_{\mathrm{sep}}\subseteq\A_{K_0}^{N}$ be the Zariski open dense subset on which the discriminant of $Q_a$ is nonzero. Define
\[
R_{\mathrm{ord}}:=\left\{\left(a,\left(r_{ij}\right)_{1\leq i<j\leq d}\right)\in U_{\mathrm{sep}}\times\A_{K_0}^{N}\colon Q_a(T)=\prod_{1\leq i<j\leq d}(T-r_{ij})\right\}.
\]
It is direct to verify that the projection $\pi:R_{\mathrm{ord}}\to U_{\mathrm{sep}}$ is finite \'etale.

Set
\[
\widetilde{X}:=\varphi^{-1}(U_{\mathrm{sep}})\times_{U_{\mathrm{sep}}}R_{\mathrm{ord}}.
\]
Let $p:\widetilde{X}\to\varphi^{-1}(U_{\mathrm{sep}})$ be the first projection. Since $\pi$ is finite \'etale, so is $p$. For $1\leq i<j\leq d$, let $\rho_{ij}(x)$ denote the multiplier of either point of the $(i,j)$-indexed two-cycle of $f_{x,t}$ given in Proposition~\ref{prop:exp}. The generic configuration $x$ then determines a point $\widetilde{x}\in\widetilde{X}(L)$ by $r_{ij}(\widetilde{x})=\rho_{ij}(x)$ for $1\leq i<j\leq d$. On $\widetilde{X}$, define
\[
\Theta_{ij}:=-\frac{w_iw_j}{2}\left(r_{ij}t-\frac{1}{tw_iw_j}+\frac{1}{w_i}+\frac{1}{w_j}\right).
\]
By Proposition~\ref{prop:exp}, one has
\[
\Theta_{ij}(\widetilde{x})=B_{ij}(x)+O(t),\quad\left.\frac{\partial\Theta_{ij}}{\partial x_r}\right|_{\widetilde{x}}=\left.\frac{\partial B_{ij}}{\partial x_r}\right|_x+O(t).
\]

By Proposition~\ref{prop:B}, there is a set $I$ of $d-2$ pairs $(i,j)$ with $1\leq i<j\leq d$ such that
\begin{equation}\label{eq:detx}
\det\left(\left.\frac{\partial B_{ij}}{\partial x_r}\right|_x\right)_{\substack{(i,j)\in I\\3\leq r\leq d}}\neq0.	
\end{equation}
Here $x$ is regarded as the generic point of $\sC_d$. It follows that
\[
\det\left(\left.\frac{\partial\Theta_{ij}}{\partial x_r}\right|_{\widetilde{x}}\right)_{\substack{(i,j)\in I\\3\leq r\leq d}}\neq0.
\]
By the definition of $\Theta_{ij}$, we see that the morphism $\widetilde{X}\to R_{\mathrm{ord}}$ has differential of rank $d-2$ at $\widetilde{x}$. Because $p$ and $\pi$ are \'etale, the differential of $\varphi$ has rank $d-2$ at $x$. Since $\dim X=d-2$, the morphism $\varphi$ is generically finite.

\smallskip

Next we prove that $\sigma_\eta$ is birational to the closure of its image. We first consider the birationality of $\varphi$ and then descend it to $E=k(\eta)$. Define
\[
Z_E=\overline{\sigma_\eta\left(\sC_{d,E}\right)}\quad\text{and}\quad Z=\left(Z_E\right)_{K_0}.
\]
We have shown that $\varphi:X\to Z$ is generically finite. After shrinking $X$ and $Z$ suitably, we may assume that $f_{u,t}^{\circ2}$ has only simple fixed points for $u\in X$ and that $\varphi:X\to Z$ is finite by \cite[Lemma~29.52.5]{Stacks}.

Via $K_0(Z)\hookrightarrow L$, base change the generic fiber of $\varphi$ to $L$, and choose an algebraic closure $\overline L$. Let $y$ be any $\overline{L}$-point of this fiber. Since the fiber is finite, $y$ is defined over a finite extension $F/L$. Equip $F$ with the unique extension of the $t$-adic valuation. Since $\varphi(y)=\varphi(x)$, the specializations of $E_2(T)$ at $x$ and $y$ are equal. Because both second iterates have only simple fixed points, we have 
\[
\left\{\rho_{f_{x,t}}(z_0)\colon z_0\in\Fix\left(f_{x,t}^{\circ2}\right)\setminus\Fix\left(f_{x,t}\right)\right\}=\left\{\rho_{f_{y,t}}(z_0)\colon z_0\in\Fix\left(f_{y,t}^{\circ2}\right)\setminus\Fix\left(f_{y,t}\right)\right\}
\]
as multisets. By Proposition~\ref{prop:noescape}, we may take a representative $(y_1,\ldots,y_d)$ of $y$ for which $y_1,\ldots,y_d\in\sO_F$ and their reductions in $\kappa_F$ are pairwise distinct. For every pair $(i,j)$ with $i<j$, applying Proposition~\ref{prop:exp} to $f_{y,t}$ gives a two-cycle associated with $(i,j)$, and let $\rho_{ij}(y)$ be the multiplier of this two-cycle. By \eqref{eq:rho}, we have
\[
t^2\rho_{ij}(y)\in\sO_F\quad\text{and}\quad\overline{t^2\rho_{ij}(y)}=\frac{1}{w_iw_j}.
\]
From $\varphi(y)=\varphi(x)$, we also deduce that
\[
\prod_{1\leq i<j\leq d}\left(V-t^2\rho_{ij}(y)\right)=\prod_{1\leq i<j\leq d}\left(V-t^2\rho_{ij}(x)\right).
\]
Since $(1/(w_iw_j))_{i<j}$ are pairwise distinct, by Lemma~\ref{lem:hensel} we see that $\rho_{ij}(y)=\rho_{ij}(x)$ for all $i<j$. Then \eqref{eq:rho} implies
\[
B_{ij}(y)-B_{ij}(x)\in\fm_F\quad(i<j).
\]
Thus,
\begin{equation}\label{eq:Bijred}
B_{ij}(\bar{y})=B_{ij}(\bar{x})\quad(i<j),
\end{equation}
where the reductions are taken in $\kappa_F$.

By Proposition~\ref{prop:B}, for the generic weight $w=(w_1,\ldots,w_d)$, the morphism
\[
\beta_w=\left(B_{ij}\right)_{i<j}:\sC_{d,k(w_1,\ldots,w_d)}\to\A^N_{k(w_1,\ldots,w_d)}
\]
is birational to the closure of its image. Therefore, the geometric generic fiber of $\beta_w$ consists of one point. Combining \eqref{eq:Bijred}, we see that $[\bar{y}]=[x]$ in $\sC_d(\kappa_F)$.

Note that $y_2-y_1\in\sO_F^\times$ is a unit. By an affine conjugacy, we may replace $(y_1,\ldots,y_d)$ by
\[
\left(\frac{y_i-y_1}{y_2-y_1}\right)_{i=1}^d,
\]
which still has integral coordinates with pairwise distinct reductions. Now we have $(y_1,y_2)=(0,1)$. Since $[\bar{y}]=[x]$, uniqueness of the normalized representative of an affine configuration shows $\bar{y}=x$ over $\kappa_F$.

Let $\mathbb D_x\subset\sC_d(F)$ be the residue polydisc consisting of normalized configurations $u=(0,1,u_3,\ldots,u_d)$ satisfying $\bar{u}=x$. Set $h=u-x$, with $h_1=h_2=0$. The construction in Proposition~\ref{prop:exp} expresses every $\Theta_{ij}(x+h)$ as a formal power series in $h_3,\ldots,h_d$ with coefficients in $\sO_F$, convergent for $(h_3,\ldots,h_d)\in\fm_F^{d-2}$. For the set $I$ chosen above, let
\[
\Theta_I=(\Theta_{ij})_{(i,j)\in I},
\]
whose Jacobian at $x$ lies in $\mathrm{GL}_{d-2}(\sO_F)$ by \eqref{eq:detx} and Proposition~\ref{prop:exp}. The equalities $\Theta_{ij}(y)=\Theta_{ij}(x)$ give
\[
0=J_{\Theta_I}(x)(y-x)+E\quad\text{and}\quad\lVert E\rVert\leq\lVert y-x\rVert^2.
\]
Multiplication by the integral inverse Jacobian yields
\[
\lVert y-x\rVert\leq\lVert y-x\rVert^2.
\]
Since $\lVert y-x\rVert<1$, we obtain $y=x$. Thus, the geometric generic fiber of $\varphi$ has exactly one geometric point. We already proved that the differential of $\varphi$ has rank $d-2$ at the generic point of $X$, so $\Omega_{K_0(X)/K_0(Z)}=0$, and hence the finite extension $K_0(X)/K_0(Z)$ is separable. It follows that $[K_0(X):K_0(Z)]=1$, and therefore $\varphi$ is birational.

Return now to the morphism $\sigma_\eta:\sC_{d,E}\to Z_E$. After replacing $Z_E$ by a nonempty Zariski open subset, we may assume that $\sigma_\eta$ is finite locally free by \cite[Lemmas 29.52.5 and 10.118.2]{Stacks}. By the above argument and \cite[Lemma~29.49.4]{Stacks}, $\sigma_\eta$ has rank one. Thus, $[E(\sC_{d,E}):E(Z_E)]=1$ and $\sigma_\eta:\sC_{d,E}\to Z_E$ is birational.

\smallskip

Finally, the first projection of $\Sigma_d(c,[x])=\left(c,\sigma_c([x])\right)$ is $c$. Thus, the extension of function fields induced by 
\[
\Sigma_d:\sX_d^{\mathrm{ord}}\to\overline{\Sigma_d\left(\sX_d^{\mathrm{ord}}\right)}
\]
is the same finite extension obtained from $\sigma_\eta$. Since the latter extension has degree one, the morphism $\Sigma_d$ is birational.
\end{proof}

\subsection{Proof of the main theorem}

\begin{proof}[Proof of Theorem~\ref{thm:main}]
Fix an integer $d\geq2$, and set $N=\binom{d}{2}$. We first prove the theorem over an algebraically closed field $k$ of characteristic different from $2$.

Define
\[
Y_k:=\overline{\tau_{d,2,k}\left(\sM_{d,k}\right)}.
\]
We prove that the induced dominant morphism $\tau_{d,2,k}:\sM_{d,k}\to Y_k$ is birational.

\smallskip

We first show that $\tau_{d,2,k}$ is generically finite. Theorem~\ref{thm:ordered} shows that $\Sigma_d:\sX_d^{\mathrm{ord}}\to Z_k$ is birational, where $Z_k:=\overline{\Sigma_d\left(\sX_d^{\mathrm{ord}}\right)}\subseteq\sI_d^{\circ}\times\A_k^N$. Then $\dim Z_k=\dim\sX_d^{\mathrm{ord}}=2d-2$. For $c=(c_0,\ldots,c_d)\in\sI_d^{\circ}(k)$, set
\[
\lambda_i(c):=1-\frac{1}{c_i}\in k\setminus\{\pm1\}\quad(0\leq i\leq d)\quad \text{and}\quad P_c(T):=\prod_{i=0}^d\left(T-\lambda_i(c)^2\right).
\]
For $a=(a_1,\ldots,a_N)\in\A^N(k)$, denote $E_a(T):=T^N+a_1T^{N-1}+\cdots+a_N$. For a monic polynomial $R(T)=T^m+b_1T^{m-1}+\cdots+b_m\in k[T]$, define
\[
\mathrm{es}_m(R):=\left(-b_1,b_2,\ldots,(-1)^m b_m\right)\in\A^m(k),
\]
the tuple of elementary symmetric functions of the roots of $R$. Define a morphism
\[
H_d:\sI_d^{\circ}\times\A_k^N\to\A_k^{d+1}\times\A_k^{d^2+1}
\]
by
\[
H_d(c,a):=\left(\mathrm{es}_{d+1}\left(\prod_{i=0}^d\left(T-\lambda_i(c)\right)\right),\mathrm{es}_{d^2+1}\left(P_c(T)E_a(T)^2\right)\right).
\]
Then we have
\[
\tau_{d,2,k}\circ\nu_d=H_d\circ\Sigma_d.
\]
Let
\[
D:=\left\{(c,a)\in\sI_d^{\circ}\times\A_k^N\colon c_i\neq c_j\text{ for }i\neq j\right\},
\]
which is a Zariski open dense subset. We claim that $H_d|_D$ is quasi-finite. Indeed, suppose $H_d(c,a)=H_d(c',a')$ with $(c,a),(c',a')\in D(k)$. Then we deduce that
\[
\prod_{i=0}^d\left(T-\lambda_i(c)\right)=\prod_{i=0}^d\left(T-\lambda_i(c')\right)\quad\text{and}\quad E_a(T)=E_{a'}(T).
\]
Since both $(\lambda_i(c))_i$ and $(\lambda_i(c'))_i$ are tuples of pairwise distinct numbers, there is a unique permutation $\pi\in\mathfrak S_{d+1}$ such that
\[
\lambda_i(c')=\lambda_{\pi(i)}(c)\quad(0\leq i\leq d).
\]
So $c_i'=c_{\pi(i)}$ for $0\leq i\leq d$. In particular, $P_{c'}(T)=P_c(T)$. From $E_a(T)=E_{a'}(T)$, we obtain $a=a'$. Thus, a fiber of $H_d|_D$ over a $k$-point contains at most $(d+1)!$ points, and $H_d|_D$ is quasi-finite. Since $\nu_d$ is dominant,
\[
Y_k=\overline{\tau_{d,2,k}\left(\nu_d\left(\sX_d^{\mathrm{ord}}\right)\right)}=\overline{H_d\left(\Sigma_d\left(\sX_d^{\mathrm{ord}}\right)\right)}=\overline{H_d(Z_k)}.
\]
Since $Z_k\cap D$ is Zariski dense in $Z_k$ and $H_d|_D$ is quasi-finite, the morphism $H_d|_{Z_k}$ is generically quasi-finite. So $\dim Y_k=\dim Z_k=2d-2=\dim\sM_{d,k}$. Therefore, $\tau_{d,2,k}:\sM_{d,k}\to Y_k$ is generically finite.

\smallskip

We next determine the geometric generic fiber of $\tau_{d,2,k}$. Take Zariski open dense subsets $\Omega\subseteq\sX_d^{\mathrm{ord}}$ and $\Omega'\subseteq Z_k$ such that $\Sigma_d|_{\Omega}:\Omega\to\Omega'$ is an isomorphism. Define $B:=\sX_d^{\mathrm{ord}}\setminus\Omega$. Since $B$ is a proper closed subset of $\sX_d^{\mathrm{ord}}$,
\[
\dim\overline{\nu_d(B)}\leq\dim B\leq2d-3.
\]
So $\overline{\nu_d(B)}$ is a proper closed subset of $\sM_{d,k}$. Because $\nu_d$ is dominant, its image contains a nonempty Zariski open subset $M_0\subseteq\sM_{d,k}$. Let $M_1\subseteq\sM_{d,k}$ be the Zariski open locus on which the $d+1$ fixed-point multipliers are pairwise distinct and $f^{\circ2}$ has only simple fixed points. By a standard specialization argument and Proposition~\ref{prop:exp}, we see that $M_1\neq\emptyset$. Define
\[
M^\circ:=\left(M_0\cap M_1\right)\setminus\overline{\nu_d(B)},
\]
which is a Zariski open dense subset of $\sM_{d,k}$. Then
\[
M^\circ\subseteq\nu_d(\sX_d^{\mathrm{ord}})\quad\text{and}\quad\nu_d^{-1}\left(M^\circ\right)\subseteq\Omega.
\]
Let $C:=\sM_{d,k}\setminus M^\circ$. Then
\[
\dim\overline{\tau_{d,2,k}(C)}\leq\dim C\leq 2d-3<2d-2=\dim Y_k.
\]
Thus, the geometric generic point of $Y_k$ does not lie in $\overline{\tau_{d,2,k}(C)}$, and every point of the geometric generic fiber of $\tau_{d,2,k}$ belongs to $M^\circ$. Let $K$ be an algebraic closure of $k(Y_k)$, and let $[f],[h]\in\sM_d(K)$ be two points of the geometric generic fiber of $\tau_{d,2,k}$. Then $[f],[h]\in M^\circ(K)$. Take
\[
\alpha=(c,[x])\in\sX_d^{\mathrm{ord}}(K)
\]
such that $\nu_d(\alpha)=[f]$. So $[g_\alpha]=[f]$ in $\sM_d(K)$. Since $[f]\in M^\circ$, we have $\alpha\in\Omega(K)$. With $(x_0,x_1,x_2)=(\infty,0,1)$, set
\[
\lambda_i=\rho_{g_\alpha}(x_i)=1-\frac{1}{c_i}\quad(0\leq i\leq d).
\]
The elements $\lambda_0,\ldots,\lambda_d$ are pairwise distinct because $[g_\alpha]\in M_1(K)$. Take $h_0:\Pone_{K}\to\Pone_{K}$ representing $[h]$. Since $S_1(h_0)=S_1(g_\alpha)$, the multisets of fixed-point multipliers coincide, so for each $i$ there exists a fixed point $z_i'\in\Fix(h_0)(K)$ satisfying $\rho_{h_0}(z_i')=\lambda_i$. As the $\lambda_i$ are pairwise distinct, the points $z_0',\ldots,z_d'$ are pairwise distinct, and hence $\Fix(h_0)=\{z_0',\ldots,z_d'\}$. Let $A\in\PGL_2(K)$ be the unique M\"obius transformation satisfying
\[
A(z_0')=\infty,\quad A(z_1')=0,\quad A(z_2')=1,
\]
and set
\[
h_1=A\circ h_0\circ A^{-1},\quad y_i=A(z_i')\quad(0\leq i\leq d).
\]
For $0\leq i\leq d$, the fixed-point index of $h_1$ at $y_i$ is
\[
\frac{1}{1-\rho_{h_1}(y_i)}=\frac{1}{1-\lambda_i}=c_i.
\]
Then Proposition~\ref{prop:nf} implies $h_1=g_{c,[y]}$. Thus,
\[
\beta:=(c,[y])\in\sX_d^{\mathrm{ord}}(K)\quad\text{and}\quad \nu_d(\beta)=[h].
\]
Since $[h]\in M^\circ$, we also have $\beta\in\Omega(K)$. From $\tau_{d,2}([f])=\tau_{d,2}([h])$, we deduce that $\Phi_{2,g_\alpha}(T)=\Phi_{2,h_1}(T)$. The two maps have the same ordered fixed-point indices $c$, so the factor $P_1(T)(c)$ in \eqref{eq:split} is the same for both. Thus,
\[
P_1(T)(c)E_2(T)(\alpha)^2=P_1(T)(c)E_2(T)(\beta)^2.
\]
Since both $E_2(T)(\alpha)$ and $E_2(T)(\beta)$ are monic, we deduce that
\[
E_2(T)(\alpha)=E_2(T)(\beta).
\]
Then $\sigma_c([x])=\sigma_c([y])$ and $\Sigma_d(\alpha)=\Sigma_d(\beta)$ by definition. Since $\alpha,\beta\in\Omega(K)$ and $\Sigma_d|_\Omega$ is injective, we have $\alpha=\beta$. Applying $\nu_d$ shows $[f]=[h]$. Thus, the geometric generic fiber of $\tau_{d,2,k}$ consists of one point.

\smallskip

We show that the finite extension $k(\sM_{d,k})/k(Y_k)$ is separable. Set
\[
L=k(Y_k)\quad\text{and}\quad F=k(Z_k).
\]
View the coordinate functions $c_0,\ldots,c_d,a_1,\ldots,a_N$ on $\sI_d^\circ\times\A_k^N$ as elements of $F$ by restriction to $Z_k$. Since $Z_k\cap D$ is Zariski dense in $Z_k$, the elements $c_0,\ldots,c_d$, and hence $\lambda_i:=1-1/c_i$ ($0\leq i\leq d$) are pairwise distinct in $F$. Note that the coefficients of $Q_c(T):=\prod_{i=0}^d(T-\lambda_i)$ belong to $L$, because they are the coordinate
functions of the first component of $H_d|_{Z_k}$ up to $\pm1$. Thus, each $\lambda_i$ is algebraic over $L$. Moreover, $Q_c(T)$ is separable over $L$. Consequently,
\[
L':=L(\lambda_0,\ldots,\lambda_d)=L(c_0,\ldots,c_d)
\]
is a finite separable extension of $L$. Similarly, we see that $P_c(T)E_a(T)^2\in L[T]$. From $Q_c(T)\in L[T]$ we see that $P_c(T)=\prod_{i=0}^d\bigl(T-\lambda_i^2\bigr)\in L[T]$. So $E_a(T)^2\in L[T]$. Set $a_0=1$. For $1\leq r\leq N$, the coefficient of $T^{2N-r}$ in $E_a(T)^2$ is
\[
2a_r+\sum_{\substack{i+j=r\\1\leq i,j\leq r-1}}a_i a_j.
\]
Since $\mathrm{char}k\neq2$, it follows inductively from $E_a(T)^2\in L[T]$ that $a_1,\ldots,a_N\in L$, i.e., $E_a(T)\in L[T]$. Since the functions $c_0,\ldots,c_d,a_1,\ldots,a_N$ generate $F=k(Z_k)$, we have
\[
F=L'=L(c_0,\ldots,c_d).
\]
In particular, $F/L$ is finite separable. Since $\Sigma_d:\sX_{d,k}^{\mathrm{ord}}\to Z_k$ is birational, we have $k(\sX_d^{\mathrm{ord}})\simeq F$. The morphism $\nu_d$ gives a tower of fields
\[
L=k(Y_k)\subseteq k(\sM_{d,k})\subseteq k(\sX_d^{\mathrm{ord}})\simeq F.
\]
Thus, the intermediate extension $k(\sM_{d,k})/L$ is separable as well.

Recall that we have shown that the geometric generic fiber of $\tau_{d,2,k}:\sM_{d,k}\to Y_k$ consists of one point. Thus, $[k(\sM_{d,k}):k(Y_k)]=1$, i.e., $\tau_{d,2,k}:\sM_{d,k}\to Y_k$ is birational.

\smallskip

By \cite[Lemma~29.52.6]{Stacks}, take a Zariski open dense subset $V\subseteq Y_k$ such that $\tau_{d,2,k}:U:=\tau_{d,2,k}^{-1}(V)\to V$ is an isomorphism. For every $x\in U(k)$, we have
\[
\tau_{d,2,k}^{-1}\left(\tau_{d,2,k}(x)\right)=\{x\}.
\]
Thus, $\tau_{d,2,k}$ is generically injective.

\medskip

Finally, let $F$ be an arbitrary field of characteristic different from $2$, and set
\[
Y_F:=\overline{\tau_{d,2,F}\left(\sM_{d,F}\right)}.
\]
Let $\overline{F}$ be an algebraic closure of $F$. Since both $\sM_{d,F}$ and $\A_F^{d+1}\times\A_F^{d^2+1}$ are affine, the ideal defining $Y_F$ is the kernel of the homomorphism on coordinate rings induced by $\tau_{d,2,F}$. The extension $F\subseteq\overline{F}$ is flat, so this kernel commutes with base change. Hence $(Y_F)_{\overline{F}}=Y_{\overline{F}}$. By the proved result over $\overline{F}$, the base change of $\tau_{d,2,F}:\sM_{d,F}\to Y_F$ to $\overline{F}$ is birational. In particular, $\tau_{d,2,F}$ is generically finite. After restricting to a nonempty Zariski open subset of $Y_F$, we may assume that it is finite locally free by \cite[Lemmas~29.52.5 and 10.118.2]{Stacks}. Let $r$ be its rank. Finite locally free rank is preserved by base change. So we have $r=1$. It follows that
\[
F(Y_F)\xrightarrow{\sim}F(\sM_{d,F}),
\]
so $\tau_{d,2,F}:\sM_{d,F}\to Y_F$ is birational. This proves the theorem over every field of characteristic different from $2$.
\end{proof}

\subsection{Characteristic two}\label{subsec:char2}
Assume throughout this subsection that $k$ is algebraically closed of characteristic $2$. We briefly indicate how the preceding argument can be modified in characteristic $2$. The fixed-index normal form, the construction of two-cycles by a multivariable Hensel's lemma, and the leading term $1/(t^2w_iw_j)$ used to label the pairs remain unchanged. The difference is that the term $2B_{ij}(x)$ in \eqref{eq:rho} vanishes, so the coefficient of $t^{-1}$ no longer contains the information on fixed-point configurations. One must retain one further term and work with the squared coordinates
\[
X_i:=x_i^2.
\]

Set
\[
V_w(X)_i=\sum_{r\neq i}\frac{w_r(w_r+w_i)}{X_i-X_r}\quad\text{and}\quad
D_{ij}(X)=(X_i-X_j)\left(V_w(X)_i-V_w(X)_j\right).
\]
The pair-cycle expansion needed in place of \eqref{eq:rho} is
\begin{equation}\label{eq:rhochar2}
\rho_{f_{x,t}}(p_{ij})
=
\frac{1}{t^2w_iw_j}
+\frac{1}{t}\left(\frac{1}{w_i}+\frac{1}{w_j}\right)
+1+\frac{D_{ij}(X)}{w_iw_j}+O(t).
\end{equation}
Here the $O(t)$ term is analytic on the same affinoid neighborhoods as in Proposition~\ref{prop:exp}. Thus, after the leading term has identified the pair $(i,j)$, the coefficient of $t^0$ recovers $D_{ij}(X)$, instead of $B_{ij}(x)$.

For $c=(c_1,\ldots,c_d)$, set
\[
R(z)=\sum_{i=1}^d\frac{c_i}{z-x_i},\quad
A_c(Z)=\sum_{i=1}^d\frac{c_i}{Z-X_i},\quad
A_{c^2}(Z)=\sum_{i=1}^d\frac{c_i^2}{Z-X_i}.
\]
If
\[
f(z)=z-\frac{1}{R(z)},
\]
then in characteristic $2$, one has
\begin{equation}\label{eq:frobchar2}
f(z)^2=F_{c,X}(z^2),\quad
F_{c,X}(Z):=Z+\frac{1}{A_{c^2}(Z)},
\end{equation}
and
\begin{equation}\label{eq:derchar2}
f'(z)=1+\frac{A_c(z^2)}{A_{c^2}(z^2)}.
\end{equation}
Consequently, for fixed $c$, every periodic multiplier depends on the fixed-point configuration only through $X=(x_1^2,\ldots,x_d^2)$.

The reconstruction statement also has a direct analogue. For generic weights $w$, the morphism
\[
\delta_w:\sC_{d,k}\to\A_k^{\binom d2},\quad
[X]\mapsto\left(D_{ij}(X)\right)_{i<j},
\]
is birational to the closure of its image. Its proof uses the same factorization argument as Proposition~\ref{prop:fac}, now applied to $M(X,V_w(X))$, together with the corresponding nonvanishing conditions that exclude the swapped factorization.

There is likewise a replacement for Lemma~\ref{lem:cl}. If $C\subseteq\{1,\ldots,d\}$ has at least two elements, define
\[
V_i^C:=\sum_{r\in C\setminus\{i\}}\frac{w_r(w_r+w_i)}{X_i-X_r},\quad
D_{ij}^C:=(X_i-X_j)(V_i^C-V_j^C),
\]
and
\[
W_C:=\sum_{i\in C}w_i,\quad
E_C:=\sum_{\substack{i<j\\i,j\in C}}w_iw_j(w_i+w_j).
\]
Then
\begin{equation}\label{eq:Dcluster}
\sum_{\substack{i<j\\i,j\in C}}w_iw_jD_{ij}^C=W_CE_C.
\end{equation}
If $C$ is proper and the $w_i,X_i$ are independent variables, then
\[
\sum_{\substack{i<j\\i,j\in C}}w_iw_jD_{ij}(X)-W_CE_C
\]
is a nonzero rational function. Its contribution from indices outside $C$ is
\[
\sum_{r\notin C}w_r
\sum_{\substack{i<j\\i,j\in C}}w_iw_j(X_i-X_j)
\left(
\frac{w_r+w_i}{X_i-X_r}-\frac{w_r+w_j}{X_j-X_r}
\right).
\]

These results give the same boundary-exclusion argument as in the main theorem. If the reduced configuration has a collision, choose a cluster at the largest valuation of a pairwise difference and rescale as in Proposition~\ref{prop:noescape}. For each internal pair, the leading term in \eqref{eq:rhochar2} determines its label. Subtracting the two universal leading terms and reducing the equality of multipliers gives equality of the corresponding reduced $D_{ij}$. Equation~\eqref{eq:Dcluster} excludes such a proper subset for a generic reference configuration. The birational reconstruction from the $D_{ij}$ then determines the reduced configuration in the variables $X_i=x_i^2$. Applying Lemma~\ref{lem:hensel} to the equations in $X_3,\ldots,X_d$, using \eqref{eq:frobchar2}--\eqref{eq:derchar2} to express the relevant multiplier functions in these variables, gives equality of the $X_i$. Since Frobenius is injective on a field, this implies equality of the $x_i$. Thus, the multiplier morphism up to level two is generically injective in characteristic $2$.

However, the function-field conclusion is different. A dominant generically finite morphism $\psi:U\to V$ between integral $k$-varieties is called \emph{generically radicial} if $k(U)/k(V)$ is purely inseparable. The generic injectivity above implies that $\tau_{d,2,k}:\sM_{d,k}\to\overline{\tau_{d,2,k}(\sM_{d,k})}$ is generically radicial. One can compute that
\[
\deg\,\tau_{d,2,k}=4^{d-2}.
\]
In particular, for $d\geq3$, $\tau_{d,2}$ is not birational, although it is generically injective on geometric points.

\section{Applications and conjectures}\label{sec:app}
\subsection{Applications}
For an integer $d\geq2$, let
\[
Y_d=\overline{\tau_{d,2}(\sM_d)}\subseteq\A^{d+1}\times\A^{d^2+1}
\]
be the Zariski closure of the image of $\tau_{d,2}$ over $\Q$.

Multiplier spectra give natural regular functions on $\sM_d$; see \cite[\S~4.5]{Sil07}. If $d=2$, then
\[
\Q[\sM_2]=\Q\left[\text{coordinate functions of }S_1\right];
\]
see \cite[\S~4.6]{Sil07}, \cite[Theorem~1.2]{Sil98}, and \cite{Mil93}. For the function fields, our main theorem shows that the first two multiplier levels generate the entire function field of $\sM_d$, for every degree $d\geq2$.
\begin{corollary}[Function-field generation]\label{cor:fcnfield}
For every integer $d\geq2$,
\[
\Q(\sM_d)=\Q\left(\text{coordinate functions of }S_1\text{ and }S_2\right).
\]
\end{corollary}
\begin{proof}
Theorem~\ref{thm:main} shows that $\tau_{d,2}:\sM_d\to Y_d$ is birational. Hence $\Q(Y_d)\simeq\Q(\sM_d)$. The function field $\Q(Y_d)$ is generated by the restrictions of the affine coordinate functions on $\A^{d+1}\times\A^{d^2+1}$, which are precisely the coordinate functions of $S_1$ and $S_2$.
\end{proof}
Note that the same conclusion holds over every field of characteristic different from $2$.
\begin{remark}
For all integers $d\geq4$ and $n\geq1$, using the techniques in \cite[\S~6.6]{Sil07} and \cite[Lemma~2.1]{Pak19}, one can prove that $\tau_{d,n,\overline{\Q}}$ is not injective. For $d=3$, direct computations show that the two rational maps
\[
g_{\pm}(z)=\frac{2z^2\left(z+16\pm\sqrt{285}\right)}{4z^2+\left(11\pm\sqrt{285}\right)z+19\pm\sqrt{285}}
\]
satisfy
\[
[g_+]\neq[g_-],\quad\tau_{3,2}([g_+])=\tau_{3,2}([g_-]),\quad\tau_{3,3}([g_+])\neq\tau_{3,3}([g_-]).
\]
Consequently, for every $d\geq3$,
\[
\overline{\Q}[\text{coordinate functions of }S_1\text{ and }S_2]\subsetneq\overline{\Q}[\sM_d].
\]
Since $\overline{\Q}/\Q$ is faithfully flat, we conclude that for every $d\geq3$,
\[
\Q[\text{coordinate functions of }S_1\text{ and }S_2]\subsetneq\Q[\sM_d].
\]
\end{remark}

\begin{corollary}
Let $d\geq3$ be an integer. The morphism
\[
S_{1,k}:\sM_{d,k}\to\A_k^{d+1}
\]
is not generically finite. Consequently, $2$ is the least integer $N$ for which $\tau_{d,N,k}$ is generically injective on $\sM_{d,k}$.
\end{corollary}
\begin{proof}
On the open locus of $f$ having only simple fixed points, let $\lambda_0,\ldots,\lambda_d$ be the multipliers of fixed points. By the fixed-point index formula \cite[Exercise~1.17]{Sil07}, we obtain
\[
\sum_{i=0}^d\prod_{j\neq i}(1-\lambda_j)-\prod_{j=0}^d(1-\lambda_j)=0,
\]
whose expression in the elementary symmetric functions is nonzero in every characteristic. Thus,
\[
\dim\overline{S_{1,k}(\sM_{d,k})}\leq d+1-1=d<2d-2=\dim\sM_{d,k};
\]
so $\tau_{d,1,k}=S_{1,k}$ is not generically finite on $\sM_{d,k}$. On the other hand, Theorem~\ref{thm:main} shows the generic injectivity for $\tau_{d,2,k}$.
\end{proof}

The following corollary is a direct consequence of Theorem~\ref{thm:main}:
\begin{corollary}
Let $d\geq2$, and let $k$ be an algebraically closed field of characteristic different from $2$. There are $2d-2$ coordinate functions among $S_1$ and $S_2$, denoted $u_1,\ldots,u_{2d-2}$, and a Zariski open dense subset $U\subseteq\sM_{d,k}$ such that
\[
u=(u_1,\ldots,u_{2d-2}):U\to\A_k^{2d-2}
\]
is \emph{\'etale}. In particular, when $k=\C$, these functions are local holomorphic coordinates at every point of $U(\C)$.
\end{corollary}
\begin{proof}
By Theorem~\ref{thm:main}, the morphism $\tau_{d,2,k}$ is birational to the closure of its image. Hence its differential has rank $2d-2$ at the generic point of $\sM_{d,k}$. Therefore, among the coordinate functions of $S_1$ and $S_2$, we can choose $2d-2$ functions $u_1,\ldots,u_{2d-2}$ whose differentials are linearly independent at the generic point. Since $k$ is perfect, the smooth locus of $\sM_{d,k}$ is nonempty and Zariski open. After shrinking to a nonempty open subset $U$ of this smooth locus on which $du_1,\ldots,du_{2d-2}$ remain linearly independent, the differential of $u$ is an isomorphism at every point of $U$. Since both $U$ and $\A_k^{2d-2}$ are smooth of the same dimension, $u$ is \'etale on $U$.
\end{proof}

\subsection{Conjectures and open questions}
\subsubsection*{Length spectra}
For a degree-$d$ rational map $f$ over $\C$ and $n\geq1$, define
\[
L_n(f)\in\R_{\geq0}^{N_{d,n}}
\]
by evaluating the elementary symmetric functions at the absolute values of the multipliers of the fixed points of $f^{\circ n}$; see \cite[Section~1.5]{JX25}. The sequence $(L_n(f))_{n\geq1}$ is called the \emph{length spectrum} of $f$. Ji and Xie proved that, outside the flexible Latt\`es locus, the full length spectrum has finite fibers \cite[Theorem~1.5]{JX23}. They conjectured that the full length spectrum map is generically injective up to complex conjugation \cite[Conjecture~1.9]{JX25}. Theorem~\ref{thm:main} suggests the following two-level strengthening.
\begin{conjecture}[Small-period length rigidity]
For every $d\geq2$, there is a proper Zariski closed subset $E_d\subsetneq\sM_d$ defined over $\R$ such that, for every $[f]\in\sM_d(\C)\setminus E_d(\C)$ and every $[g]\in\sM_d(\C)$,
\[
L_1(f)=L_1(g)\quad\text{and}\quad L_2(f)=L_2(g)
\]
imply
\[
[g]=[f]\quad\text{or}\quad[g]=[\overline{f}].
\]
Here $[\overline{f}]$ denotes the conjugacy class obtained by applying complex conjugation to the coefficients of a representative of $[f]$.
\end{conjecture}

\subsubsection*{Combinations of multipliers}
\begin{question}[Two levels]
For integers $d\geq2$ and $n_2>n_1\geq1$, is the morphism
\[
(S_{n_1},S_{n_2}):\sM_{d,k}\to\A^{N_{d,n_1}}_{k}\times\A^{N_{d,n_2}}_{k}
\]
generically injective?
\end{question}
Theorem~\ref{thm:main} shows that $\tau_{d,2,k}=(S_1,S_2)$ is generically injective. See \cite[\S~5]{Zha26} for some related questions.

\subsubsection*{Non-injective locus}
For $d\geq2$ and $m\geq1$, define
\[
R_{d,m,\C}:=\sM_{d,\C}\times_{\prod_{j=1}^m\A_{\C}^{N_{d,j}}}\sM_{d,\C}
\]
with the reduced structure, where both morphisms to the product are $\tau_{d,m}$. The sequence $(R_{d,m,\C})_{m\geq1}$ forms a descending chain of closed subsets of $\sM_{d,\C}\times\sM_{d,\C}$. By Noetherianity,
\[
R_{d,\infty,\C}:=\bigcap_{j\geq1}R_{d,j,\C}=R_{d,m,\C}
\]
for all sufficiently large $m>0$.

If $f=h_1\circ h_2$ with $\deg h_1,\deg h_2\geq2$, then $h_2\circ h_1$ is an \emph{elementary transformation} of $f$. Ji and Xie conjecture that Latt\`es maps and the equivalence generated by elementary transformations are the only obstructions to the injectivity of the full multiplier spectrum morphism; see \cite[Conjecture~1.7]{JX25}.
\begin{question}
Let $\mathcal{L}_d\subseteq\sM_{d,\C}$ be the locus of Latt\`es maps, and let $\Delta_d\subseteq\sM_{d,\C}\times\sM_{d,\C}$ be the diagonal. For $d\geq2$, classify the positive-dimensional irreducible components $Z$ of $R_{d,2,\C}$ satisfying
\[
Z\not\subseteq\Delta_d\quad\text{and}\quad Z\not\subseteq\mathcal L_d\times\mathcal L_d.
\]
\end{question}

\subsection*{Acknowledgment}
The author would like to thank Junyi~Xie, Valentin~Huguin, and Igors~Gorbovickis for helpful comments on a preliminary version of the paper.
\subsection*{AI Disclosure}
A first proof of the characteristic zero case of Theorem~\ref{thm:main} was found by the Rethlas system (see \cite{Rethlas}) with OpenAI's models. The author then independently checked the argument, corrected several errors and gaps, modified the proof, and rewrote the manuscript in full. The author takes full responsibility for all statements, proofs, and conclusions in the paper.

\end{document}